\documentclass[10pt]{amsart}
\usepackage{latexsym}
\usepackage[mathscr]{euscript}
\usepackage{amsthm}
\usepackage{amssymb}
\usepackage{xspace}
\usepackage{amsmath}
\usepackage[all]{xy}
\usepackage{amsmath,amssymb,epsfig}

\newtheorem{theorem}{Theorem}
\newtheorem{lemma}[theorem]{Lemma}

\title{A Lattice of Finite-Type Invariants of Virtual Knots}
\author{Micah W. Chrisman}
\begin{document}
\begin{abstract} We construct an infinite commutative lattice of groups whose dual spaces give Kauffman finite-type invariants of long virtual knots. The lattice is based ``horizontally'' upon the Polyak algebra and extended ``vertically'' using Manturov's functorial map $f$. For each $n$, the $n$-th vertical line in the lattice contains an infinite dimensional subspace of Kauffman finite-type invariants of degree $n$.  Moreover, the lattice contains infinitely many inequivalent extensions of the Conway polynomial to long virtual knots, all of which satisfy the same skein relation. Bounds for the rank of each group in the lattice are obtained. 
\end{abstract}
\maketitle
\section{Introduction} \label{intro}
\subsection{Overview} In \cite{GPV}, it was shown that there exists a sequence of finitely generated abelian groups $\vec{\mathscr{P}}_t$ and surjections $\vec{\mathscr{P}}_t \to \vec{\mathscr{P}}_{t-1}$:
\[
\xymatrix{
\ldots \ar[r] & \vec{\mathscr{P}}_t \ar[r] & \vec{\mathscr{P}}_{t-1}\ar[r] & \ldots \ar[r] & \vec{\mathscr{P}}_2 \ar[r] & \vec{\mathscr{P}}_1},
\]
such that the elements of $\text{Hom}_{\mathbb{Z}}(\vec{\mathscr{P}}_t,\mathbb{Q})$ are Kauffman finite-type invariants of long virtual knots of degree $\le t$. The sequence contains many interesting \emph{classical} knot invariants.  For example, the Conway polynomial has several combinatorial formulae which lie in this group.
   
\begin{equation} \label{commlatt}
\begin{array}{c}
\xymatrix{
\ldots \ar[r] & \ar[d]_{d_t[\infty \to m]} \vec{\mathscr{X}}_t[\infty] \ar[r] & \ar[d]_{d_{t-1}[\infty \to m]} \vec{\mathscr{X}}_{t-1}[\infty] \ar[r] & \ldots \ar[r] & \ar[d]_{d_2[\infty \to m]} \vec{\mathscr{X}}_2[\infty] \ar[r] & \ar[d]_{d_1[\infty \to m]} \vec{\mathscr{X}}_1[\infty] \\
\cdots & \ar[d]_{d_t[3 \to 2]} \vdots & \ar[d]_{d_{t-1}[3 \to 2]} \vdots & \cdots & \ar[d]_{d_2[3 \to 2]} \vdots & \ar[d]_{d_1[3 \to 2]}  \vdots \\
\ldots \ar[r] & \ar[d]_{d_t[2 \to 1]} \vec{\mathscr{X}}_t[2] \ar[r] & \ar[d]_{d_{t-1}[2\to 1]} \vec{\mathscr{X}}_{t-1}[2] \ar[r] & \ldots \ar[r] & \ar[d]_{d_2[2 \to 1]} \vec{\mathscr{X}}_2[2] \ar[r] & \ar[d]_{d_1[2 \to 1]} \vec{\mathscr{X}}_1[2] \\ 
\ldots \ar[r] & \vec{\mathscr{X}}_t[1] \ar[r] & \vec{\mathscr{X}}_{t-1}[1] \ar[r] & \ldots \ar[r] & \vec{\mathscr{X}}_2[1] \ar[r] & \vec{\mathscr{X}}_1[1] \\
}
\end{array}
\end{equation}

Many examples of finite-type invariants are beyond description of these groups \cite{C1,C2,CM}. In \cite{CM}, the sequence of Polyak groups was extended by parity to a sequence of Kauffman finite-type invariants which contain many invariants which are not of Goussarov-Polyak-Viro finite-type. Each of the groups has finite rank. In the present paper, we construct a commutative \emph{lattice} of groups $\vec{\mathscr{X}}_t[m]$ with surjective arrows (see Equation \ref{commlatt}). The lattice satisfies the following properties (denoted throughout as Properties \ref{prop1}-\ref{prop5}).

\begin{enumerate}
\item \label{prop1} The elements of the dual space $\text{Hom}_{\mathbb{Z}}(\vec{\mathscr{X}}_t[m],\mathbb{Q})$ yield Kauffman finite-type invariants of degree $\le t$.

\item \label{prop2} The group $\vec{\mathscr{X}}_t[m]$ contains an isomorphic copy of the Polyak group $\vec{\mathscr{P}}_t$ which determines the value of an invariant in $\text{Hom}_{\mathbb{Z}}(\vec{\mathscr{X}}_t[m],\mathbb{Q})$ on the set of classical knots.

\item \label{prop3} The lattice contains combinatorial representations of infinitely many inequivalent extensions of the Conway polynomial to long virtual knots. 

\item \label{prop4} For each $t$, the $t$-th column in the lattice contains combinatorial representations of an infinite dimensional subspace of finite-type invariants of degree $t$.

\item \label{prop5} The rank of $\rho_t[m]$ of $\vec{X}_t[m]$ satisfies:
\[
\frac{t+1}{m} {m+t \choose 1+t} \le \rho_t[m] \le \Omega_t[m],
\]
where $\Omega_t[m]$ is given in Section \ref{omegatm}.
\end{enumerate}

The groups $\vec{\mathscr{X}}_t[m]$ in the commutative lattice are constructed as a \emph{generalized Polyak group} \cite{GPV, CM}.  In other words, it is a quotient of a free abelian group of labelled Gauss diagrams by some relations which roughly correspond to the sum of subdiagrams of all Reidemeister relations.  In our case, the labels will come from iterates of Manturov's functorial map $f$.

The organization of this paper is as follows.  In the remainder of Section \ref{intro}, we review virtual knot theory which is relevant to the present paper. In Section \ref{fmlabel}, we define the groups in the lattice, show that it is commutative, and establish Properties \ref{prop1} and \ref{prop2}.  In Section \ref{fmcon}, we define the extensions of the Conway polynomial and verify Property \ref{prop3}.  In Section \ref{prop4sec}, we establish Property \ref{prop4}.  Finally, in Section \ref{prop5sec}, we prove the bounds on the rank of the lattice groups given in Property \ref{prop5}.
 
\subsection{Acknowledgements} The idea to use $f^m$-labelling to create combinatorial formulae was suggested to the author by V.O. Manturov. The author is also indebted to him for the properties of the map $f$, which are used throughout. Also, the author would like the thank him for a several careful readings of earlier drafts of this paper. This paper was originally titled ``Combinatorial Formulae for Finite-Type Invariants of Virtual Knots" and was presented at Knots in Poland III.  This version is a substantial rewrite reflecting conversations with A. Gibson. The author is grateful for his interest in this work.  In addition, M. Polyak and H. Morton asked questions after and during (respectively) the presentation.  The answers to those questions is contained herein.

\subsection{Background} Let $\mathscr{D}$ denote the set of Gauss diagrams on $\mathbb{R}$ or the set of Gauss diagrams on $S^1$.  In diagrammatic form, the Reidemeister moves may be written as in Figure \ref{gaussmoves}.  Here, the total number of necessary Reidemeister moves has been reduced via \"{O}stlund's theorem \cite{Ost} as in \cite{GPV}. Two Gauss diagrams $D$, $D'$ are said to be Reidemeister equivalent if there is a sequence $D=D_1 \leftrightarrow D_2 \leftrightarrow \ldots \leftrightarrow D_n=D'$ of Reidemeister moves transforming $D$ into $D'$.  
\begin{figure}[h]
\[
\begin{array}{cc} \xymatrix{ \begin{array}{c}\scalebox{.25}{\psfig{figure=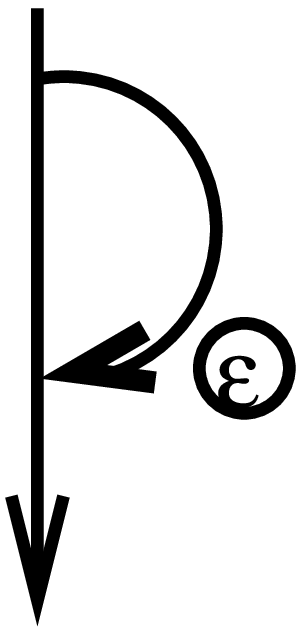}} \end{array} \ar[r]^{\Omega 1} & \ar[l] \begin{array}{c}\scalebox{.25}{\psfig{figure=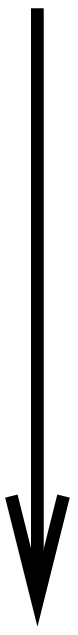}}\end{array} } & \hspace{.5cm} \xymatrix{ \begin{array}{c}\scalebox{.25}{\psfig{figure=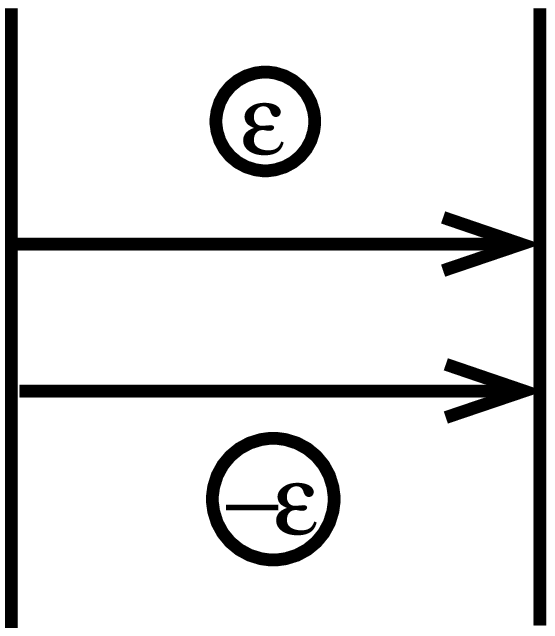}} \end{array} \ar[r]^{\Omega 2} & \ar[l] \begin{array}{c} \scalebox{.25}{\psfig{figure=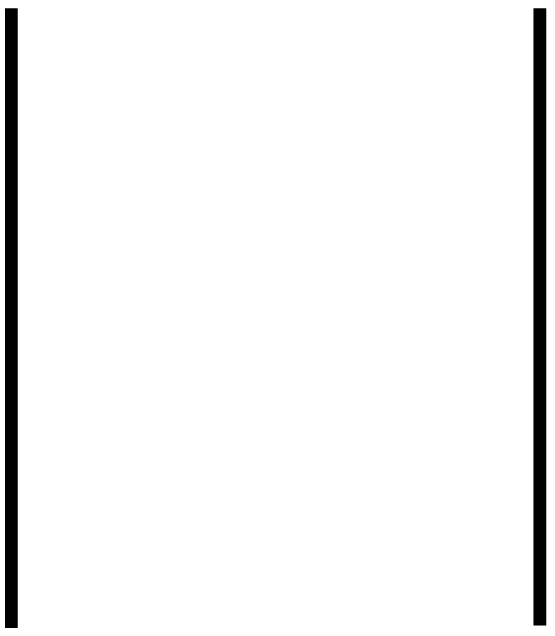}} \end{array}}
\end{array}
\]
\[  
\begin{array}{c}
\xymatrix{ \begin{array}{c} \scalebox{.25}{\psfig{figure=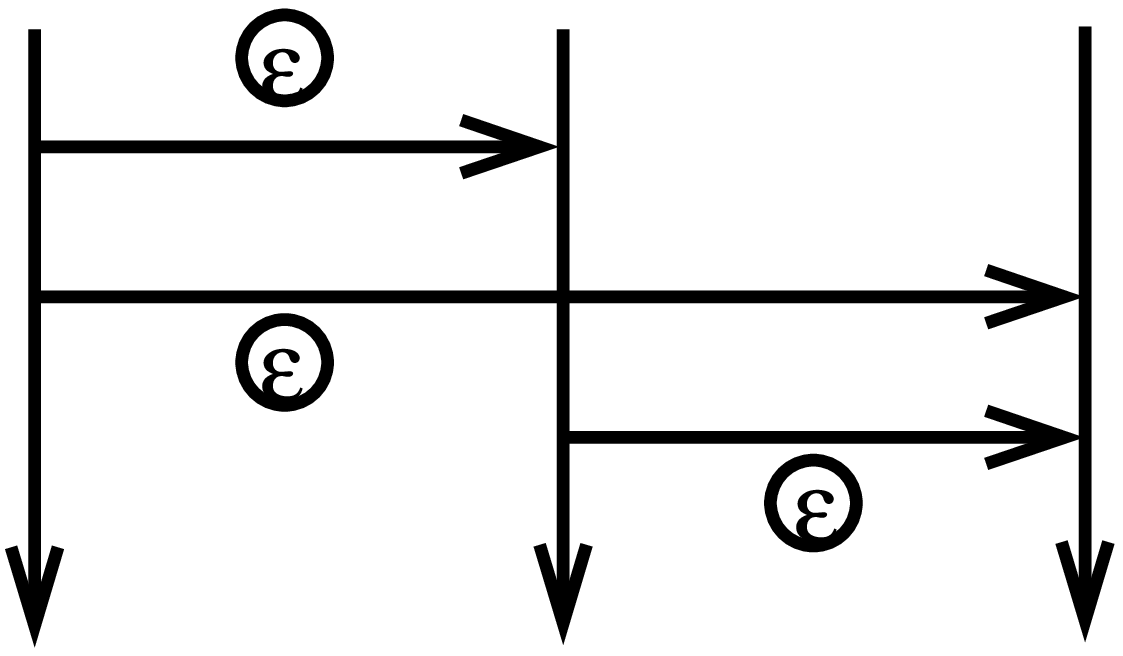}} \end{array} \ar[r]^{\Omega 3} & \ar[l] \begin{array}{c}
\scalebox{.25}{\psfig{figure=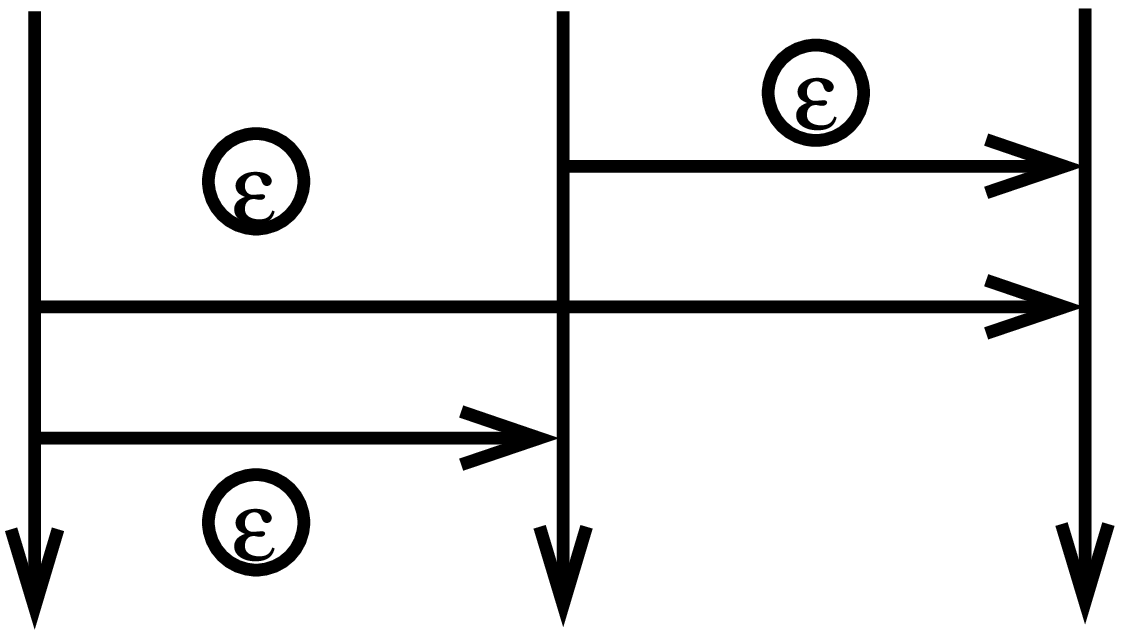}}\end{array} } \end{array}
\]
\caption{Sufficient set of  Reidemeister moves in Gauss diagram notation}\label{gaussmoves}
\end{figure}

The lattice is constructed using parity \cite{Ma01,Ma2,Ma6,Ma7,Ma8,IMN}. Let $D$ be a Gauss diagram. Let $C(D)$ denote the set of arrows of $D$. If $D \leftrightharpoons D'$ is a Reidemeister move, then there is a one-to-one correspondence between arrows not involved in the move. For $w \in C(D)$, we denote the corresponding unaffected arrow as $w' \in C(D')$. Let $\mathscr{D}^{(1,0)}$ denote the set of Gauss diagrams where each arrow is labelled with an element of $\mathbb{Z}_2=\{0,1\}$.  A \emph{parity} is a function $P:\mathscr{D} \to \mathscr{D}^{(1,0)}$ satisfying the following four properties.
\begin{enumerate}
\item If $D \in \mathscr{D}$ has an arrow $x$ with consecutive endpoints then $P$ assigns the label $0$ to $x$.
\item If $D \in \mathscr{D}$ and $x,y \in C(D)$ have opposite sign and are embedded as the two affected arrows in a Reidemeister 2 move, then $P$ assigns the same label to $x$ and $y$. 
\item Suppose that $D \rightleftharpoons D'$ is a Reidemeister 3 move. Let $\{x,y,z\}$  denote the set of arrows of $D$ which are changed by the move and $\{x',y',z'\}$ the set of corresponding arrows in $D'$. Then $P$ assigns the label $1$ to either zero or two elements of $\{x,y,z\}$. If $t \in \{x,y,z\}$, then $P$ assigns the same label to $t$ and $t'$ in $\{x',y',z'\}$.
\item If $D \rightleftharpoons D'$ is any Reidemeister move, and $(y,y')$ is a corresponding pair of unaffected arrows, then $P$ assigns the same label to $y$ and $y'$.      
\end{enumerate}

The standard example of a parity is the \emph{Gaussian parity}.  Let $D$ be a Gauss diagram. To $D$ we associate its \emph{intersection graph}. Two arrows $a$ and $b$ are said to {\em intersect} (or to be {\em linked}) if their endpoints alternate on $\mathbb{R}$ or $S^1$. We write $(a,b)=(b,a)=1$ if $a$ and $b$ intersect and $(a,b)=0$ otherwise. The intersection graph is the graph with a vertex for each arrow of the diagram and an edge between two vertices $a$ and $b$ exactly when $(a,b)=1$

Given a Gauss diagram $D$ and its intersection graph $G$, the Gaussian parity is defined as follows.  An arrow in $D$ is labelled $1$ if the degree of its vertex in $G$ is odd and a $0$ if the degree of its vertex in $G$ is even. It is easy to see that this definition satisfies the parity axioms.

We will say that $P$ is a \emph{parity of flat virtual knots} if for all diagrams $D$, $D'$ such that $D'$ is obtained from $D$ be changing the direction of an arrow, then $P$ assigns the same label to corresponding arrows of $D$ and $D'$.  For example, the Gaussian parity is a parity of flat virtual knots.

Lastly, we will need the functorial map $f:\mathbb{Z}[\mathscr{D}] \to \mathbb{Z}[\mathscr{D}]$ due to Manturov (see e.g. \cite{CM}). Let $P$ be any parity. $f(D)$ is defined to be the Gauss diagram which deletes all arrows in $D$ which are odd with respect to $\mathscr{P}$. We note that if $P(D)$ has all arrows marked $0$ then $f(D)=D$. Also note that if $D$ and $D'$ are related by a Reidemeister move, then either $f(D)=f(D')$ or $f(D)$ and $f(D')$ are related by a Reidemeister move.

\subsubsection{Finite-Type Invariants of Virtual Knots} There are two two notions of finite-type invariants of virtual knots. The first type is the natural generalization of Vassiliev invariants to virtual knots. Finite-type invariants such as these were first studied by Kauffman. Therefore, we say that an invariant of virtual knots is said to be of \emph{Kauffman finite-type of degree} $\le n$ if it vanishes on all diagrams having more than $n$ graphical vertices \cite{KaV}.  Graphical vertices are defined via the following filtration:
\[
\begin{array}{c} \scalebox{.18}{\psfig{figure=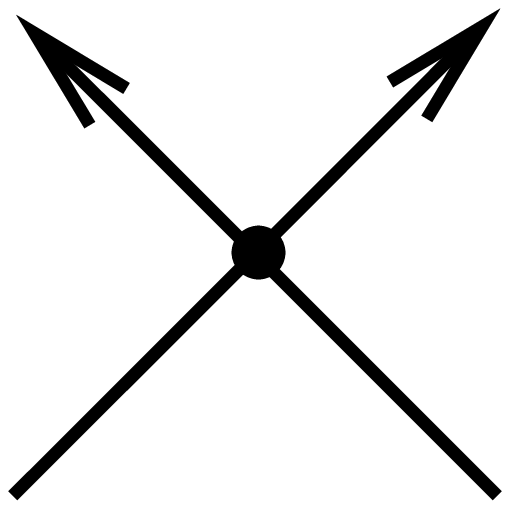}} \\ \end{array}:= \begin{array}{c} \scalebox{.18}{\psfig{figure=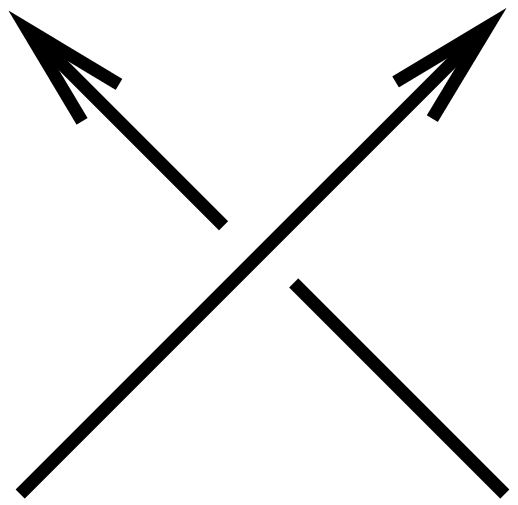}} \\ \end{array}-\begin{array}{c} \scalebox{.18}{\psfig{figure=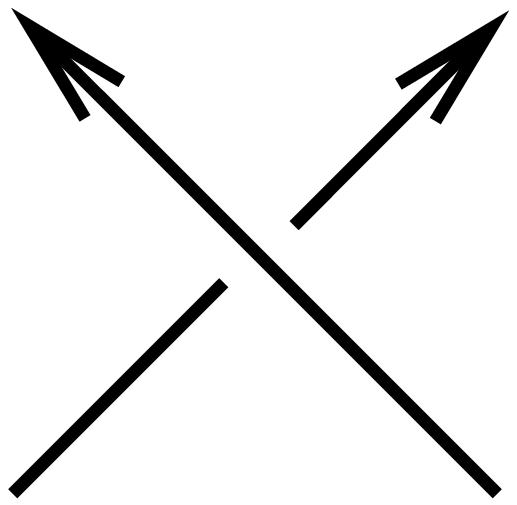}} \\\end{array}.
\]

The second kind of finite-type invariants of virtual knots arises from the \emph{Polyak groups}. This notion was originally studied by Goussarov, Polyak, and Viro. Let $\vec{\mathscr{A}}$ denote the set of Gauss diagrams with arrows drawn as dashed lines. Let $\vec{A}^t$ denote those diagrams having more than $t$ arrows. Let $\mathbb{Z}[\vec{\mathscr{A}}]$ denote the free abelian group generated by $\vec{\mathscr{A}}$. The Polyak algebra \cite{GPV} has relations given in Figure \ref{polyak}.

\begin{figure}
\[
\underline{\text{P1}:}
\,\,\,\begin{array}{c}\scalebox{.15}{\psfig{figure=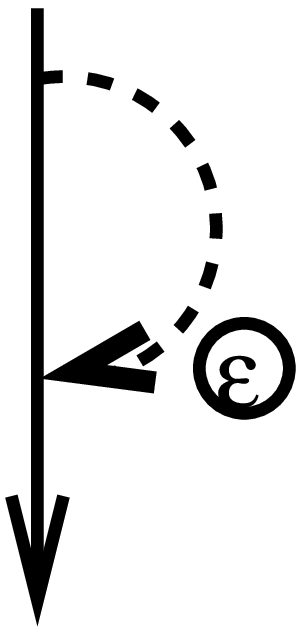}}
\end{array} =0,\,\,\, \underline{\text{P2}:}
\,\,\,\begin{array}{c}\scalebox{.15}{\psfig{figure=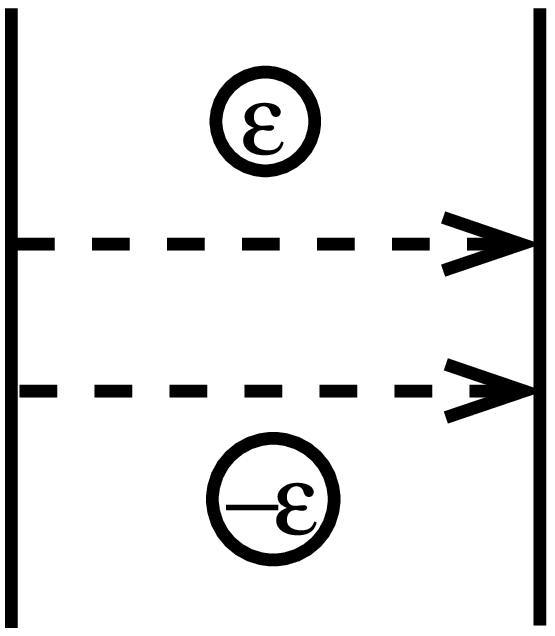}}
\end{array}+\begin{array}{c}\scalebox{.15}{\psfig{figure=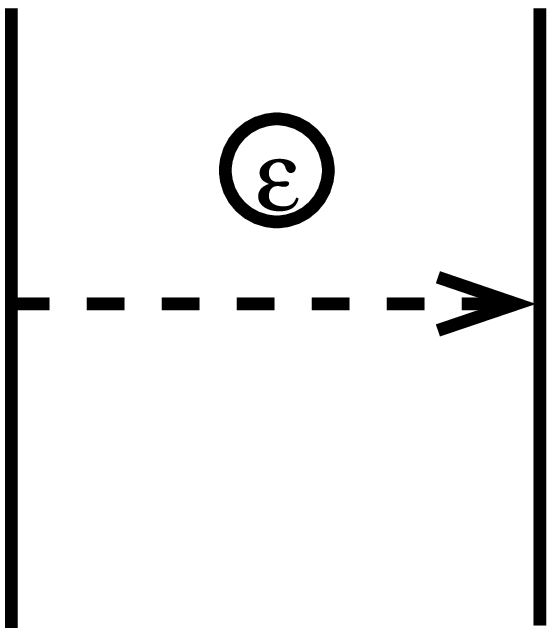}}
\end{array}+\begin{array}{c}\scalebox{.15}{\psfig{figure=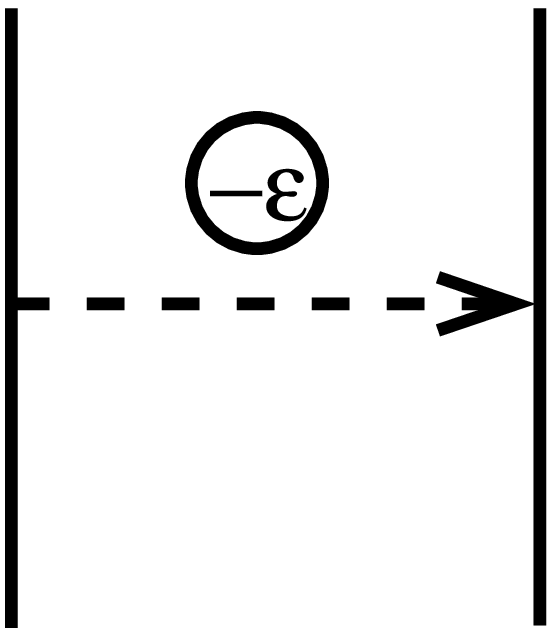}}
\end{array}=0,
\]
\begin{eqnarray*}
\underline{\text{P3}:} \,\,\,\begin{array}{c}\scalebox{.15}{\psfig{figure=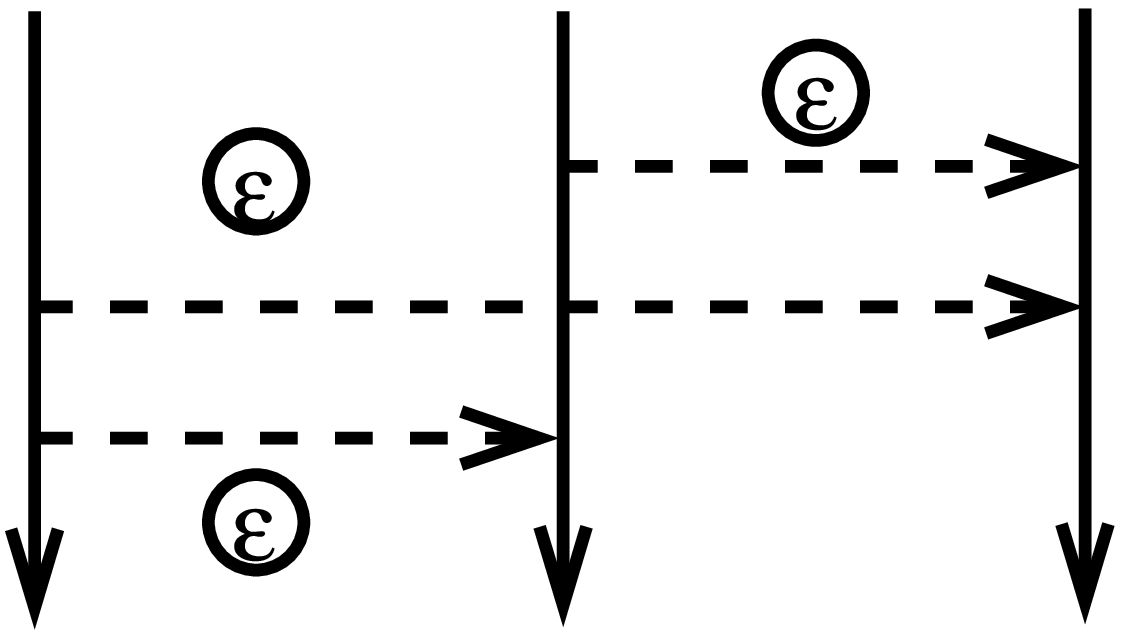}} \end{array}+\begin{array}{c}\scalebox{.15}{\psfig{figure=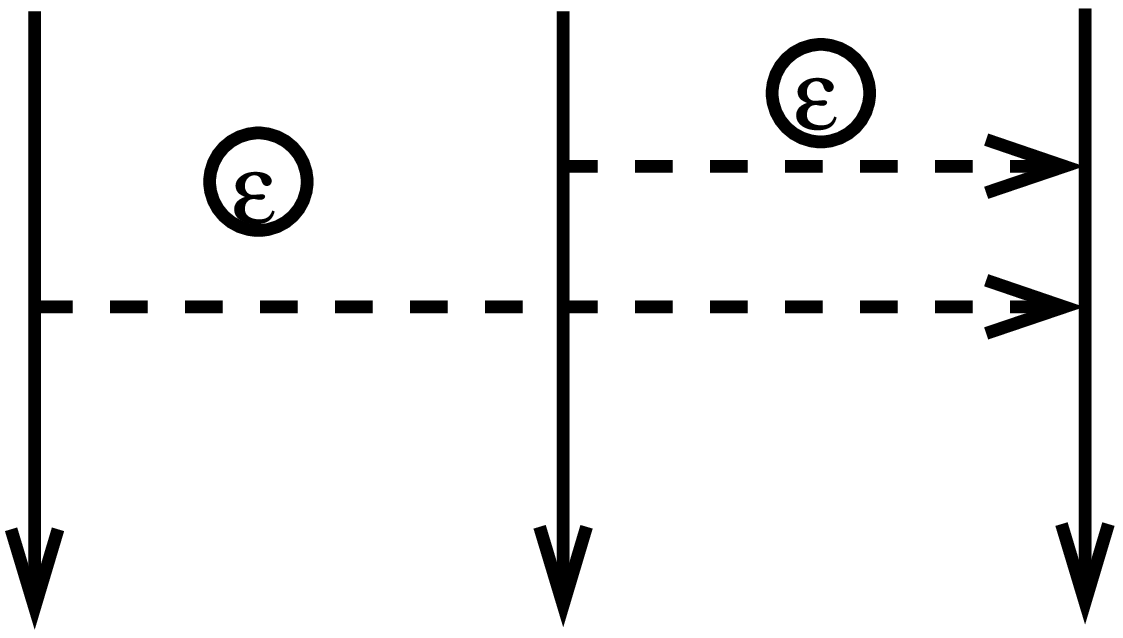}} \end{array}+\begin{array}{c}\scalebox{.15}{\psfig{figure=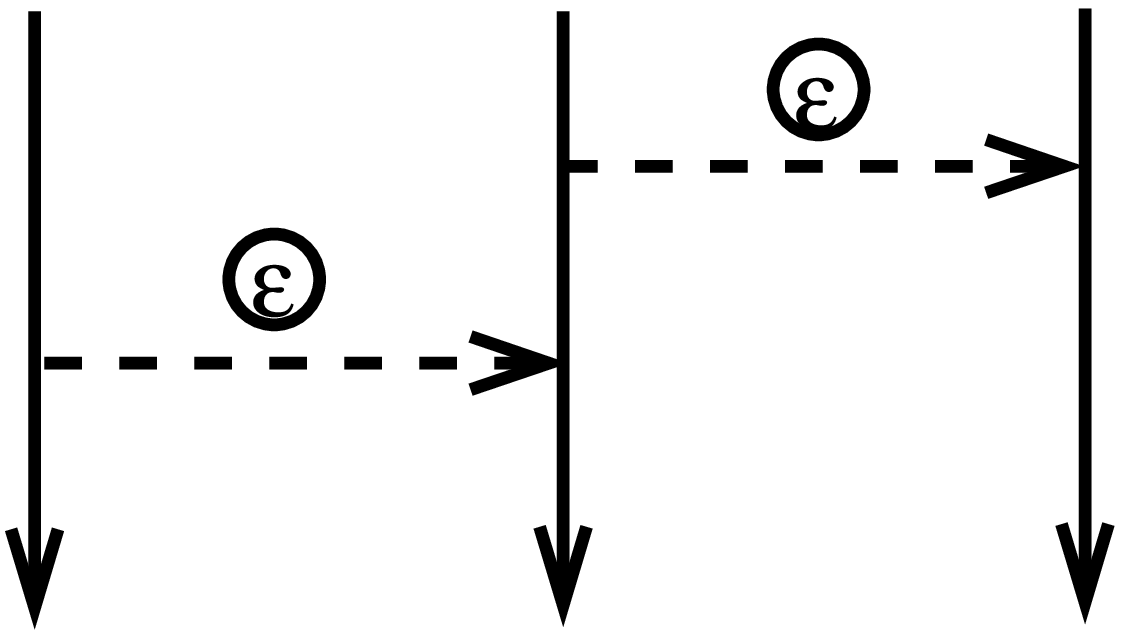}} \end{array}+\begin{array}{c}\scalebox{.15}{\psfig{figure=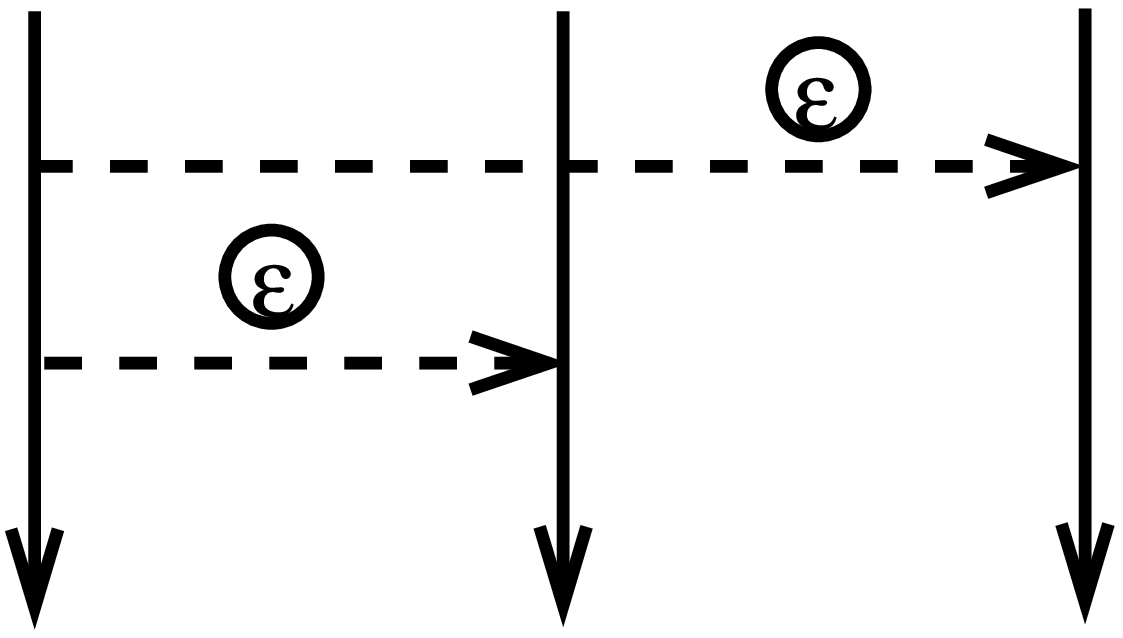}} \end{array} &=& \\ \begin{array}{c}\scalebox{.15}{\psfig{figure=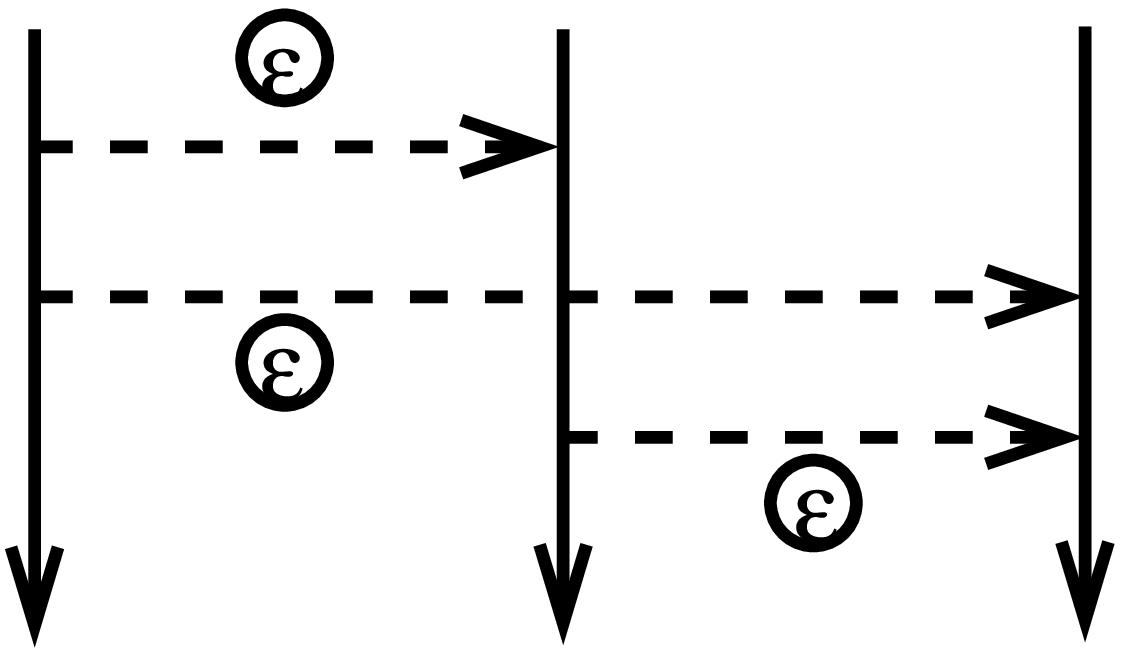}} \end{array}+\begin{array}{c}\scalebox{.15}{\psfig{figure=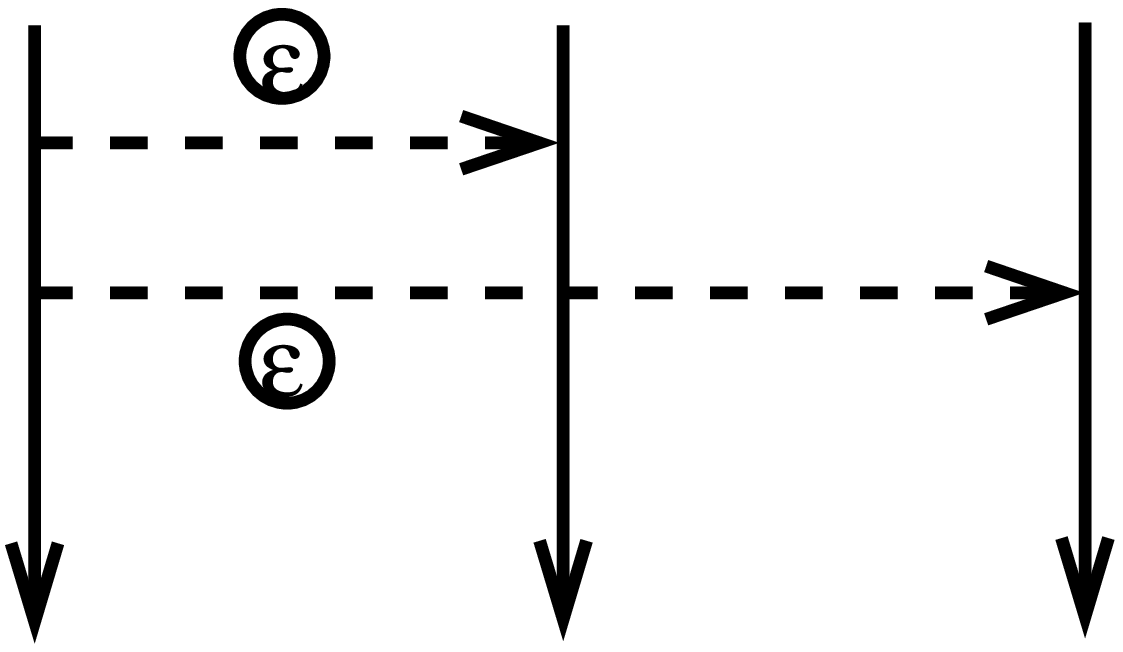}} \end{array}+\begin{array}{c}\scalebox{.15}{\psfig{figure=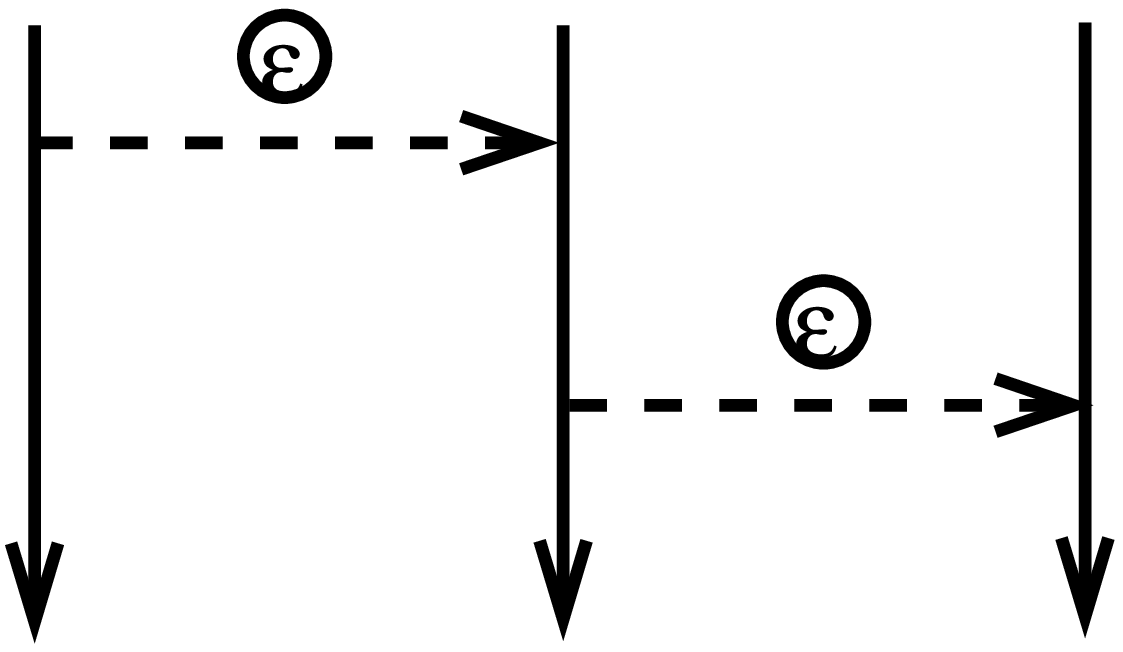}} \end{array}+\begin{array}{c}\scalebox{.15}{\psfig{figure=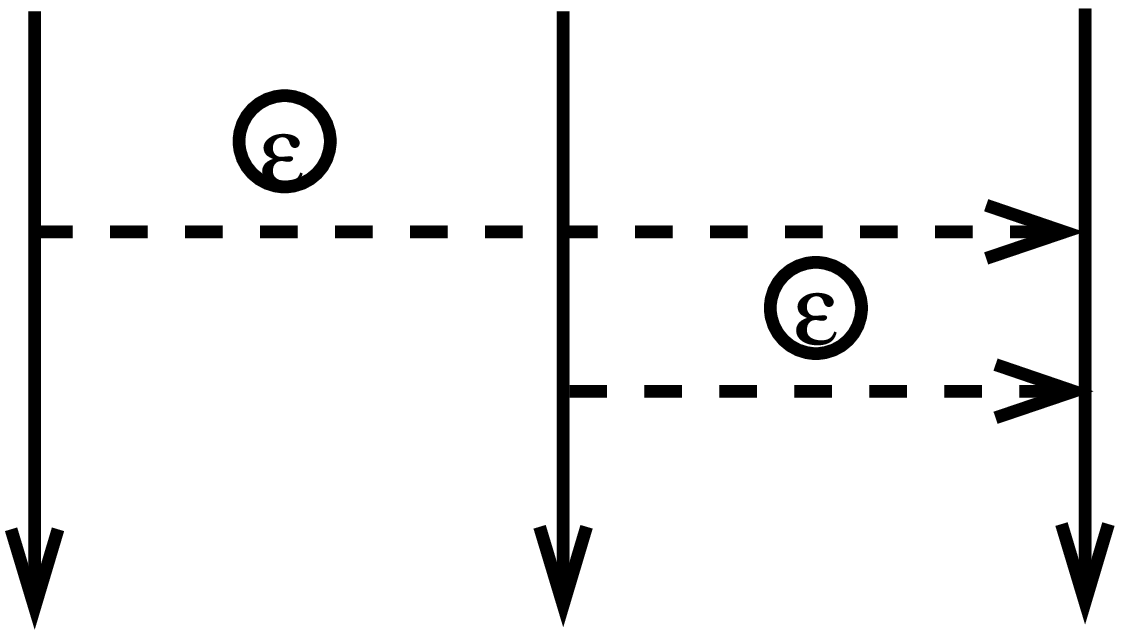}} \end{array} & & \\
\end{eqnarray*}
\caption{Polyak Relations} \label{polyak}
\end{figure}

The \emph{Polyak group} is given by the quotient:
\[
\vec{\mathscr{P}}_t=\frac{\mathbb{Z}[\vec{\mathscr{A}}]}{\left<\vec{\text{P1}},\vec{\text{P2}},\vec{\text{P3}},\vec{A}_t \right>}.
\]

The virtual knot invariants that arise from the Polyak groups are defined using the \emph{subdiagram map} $I$:
\[
I(D)=\sum_{D' \subset D} i(D'),
\] 
where $i$ makes every arrow of $D'$ dashed and the sum is taken over all subdiagrams of $D$.  If $v \in \text{Hom}_{\mathbb{Z}}(\mathscr{P}_t,\mathbb{Q})$, then $v \circ I$ is a virtual knot or virtual long knot invariant. Moreover, $v \circ I$ is a Kauffman finite-type invariant of degree $\le t$ \cite{GPV}.

Not all Kauffman finite-type invariants are represented by these groups. Those invariants which factor through $\vec{\mathscr{P}}_t$ are said to be of \emph{Goussarov-Polyak-Viro finite-type}.

\section{Properties \ref{prop1},\ref{prop2}: The Lattice of Finite-Type Invariants} \label{fmlabel} The present section defines the $f^m$-labelled Polyak groups and shows how they form the commutative lattice given in Equation \ref{commlatt}.  In addition, Properties \ref{prop1} and \ref{prop2} are established.

\subsection{Definition and Commutativity of the Lattice} Let $P$ be any parity. Let $m \in \mathbb{N} \cup \{\infty\}$. For a Gauss diagram $D$, consider $P(D)$ (i.e. the diagram $D$ with arrows labelled as prescribed by $P$).  For an arrow $x \in C(D)$, let $i$ be the smallest number $1 \le i \le m$ such that $x \notin C(f^i(D))$.  If $x \in C(f^i(D))$ for all $i$, $1 \le i \le m$ set $i=m+1$ (or $i=\infty$ if $m=\infty$).  The \emph{label} of the arrow $x$ is the natural number $i$.  A labelling of a Gauss diagram according to this procedure will be called an $f^m$-\emph{labelling}. The $f^m$-labelling of a Gauss diagram $D$ satisfies the following properties:
\begin{enumerate}
\item The label of an isolated arrow is $m+1$.
\item The label of two arrows involved in an $\Omega 2$ move are identical. Deleting the two arrows in the move does not affect the $f^m$-labelling of the other arrows in the diagram.
\item The labels $\{i,j,k\}$ of the corresponding arrows on LHS and RHS of an $\Omega 3$ move are the same. Also, the labels satisfy one of the relations: $i>j=k$, $j>i=k$, $k>i=j$, $i=j=k=m+1$.
\end{enumerate}
Let $\vec{\mathscr{A}}[m]$ denote the set of signed arrow diagrams where all of the arrows are arbitrarily labelled from $1$ to $m+1$ and the arrows are drawn formally \emph{dashed}. We define the map $\Lambda[m]:\mathbb{Z}[\mathscr{D}] \to \mathbb{Z}[\vec{\mathscr{A}}[m]]$ to be the map which assigns the $f^m$-labelling to each Gauss diagram and makes all of the arrows dashed. Define the map $I[m]:\mathbb{Z}[\vec{\mathscr{A}}[m]] \to \mathbb{Z}[\vec{\mathscr{A}}[m]]$ by: 
\[
I[m](D)=\sum_{F \subset D} F,
\]
where the sum is over all subdiagrams of $D$.  Note that the label and sign of each arrow is preserved in the subdiagram. 

We define some relations on $\mathbb{Z}[\vec{\mathscr{A}}[m]]$ as follows:
\[
\underline{\vec{\text{Q1}}[m]}:\,\,\, \begin{array}{c} \scalebox{.28}{\psfig{figure=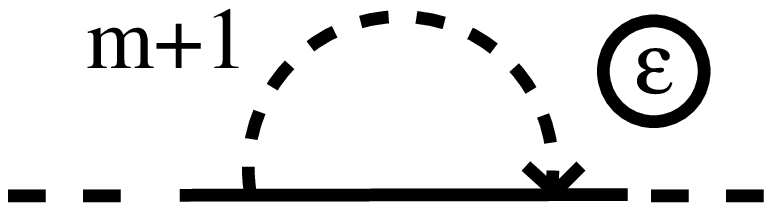}} \end{array}=0,\,\,\,\, \underline{\vec{\text{Q2}}[m]}:\,\,\, \begin{array}{c} \scalebox{.23}{\psfig{figure=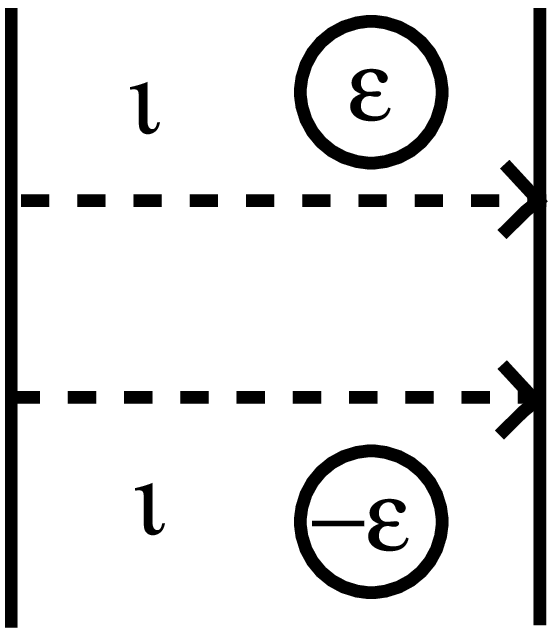}} \end{array}+\begin{array}{c} \scalebox{.23}{\psfig{figure=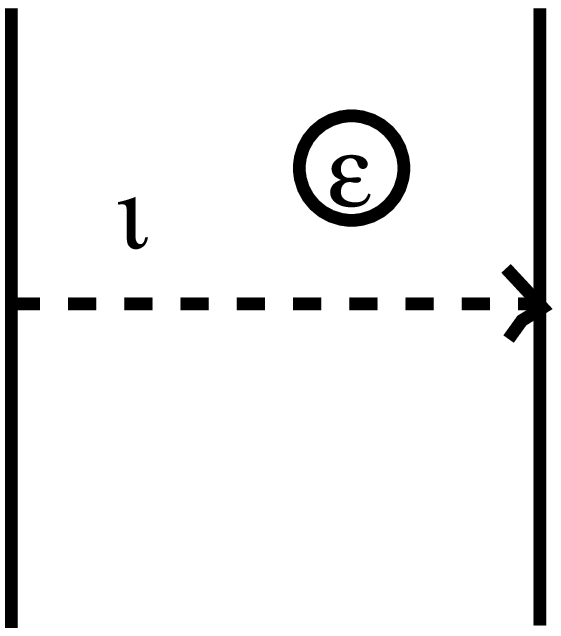}} \end{array}+\begin{array}{c} \scalebox{.23}{\psfig{figure=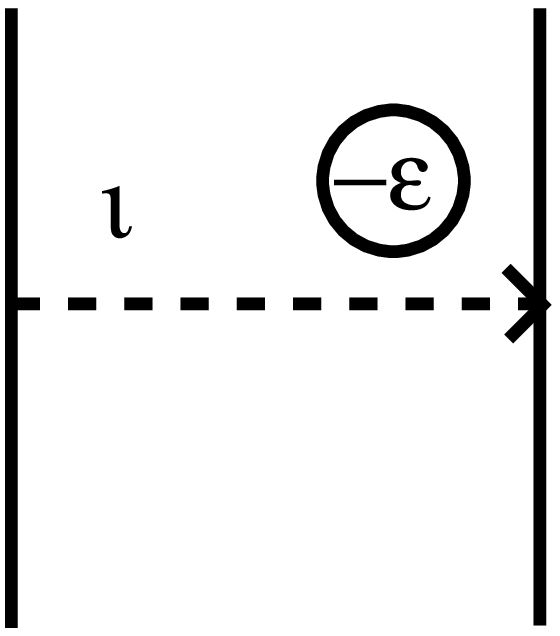}} \end{array}=0
\]
\begin{eqnarray*}
\underline{\vec{\text{Q3}}[m]}:\,\,\,  \begin{array}{c}\scalebox{.23}{\psfig{figure=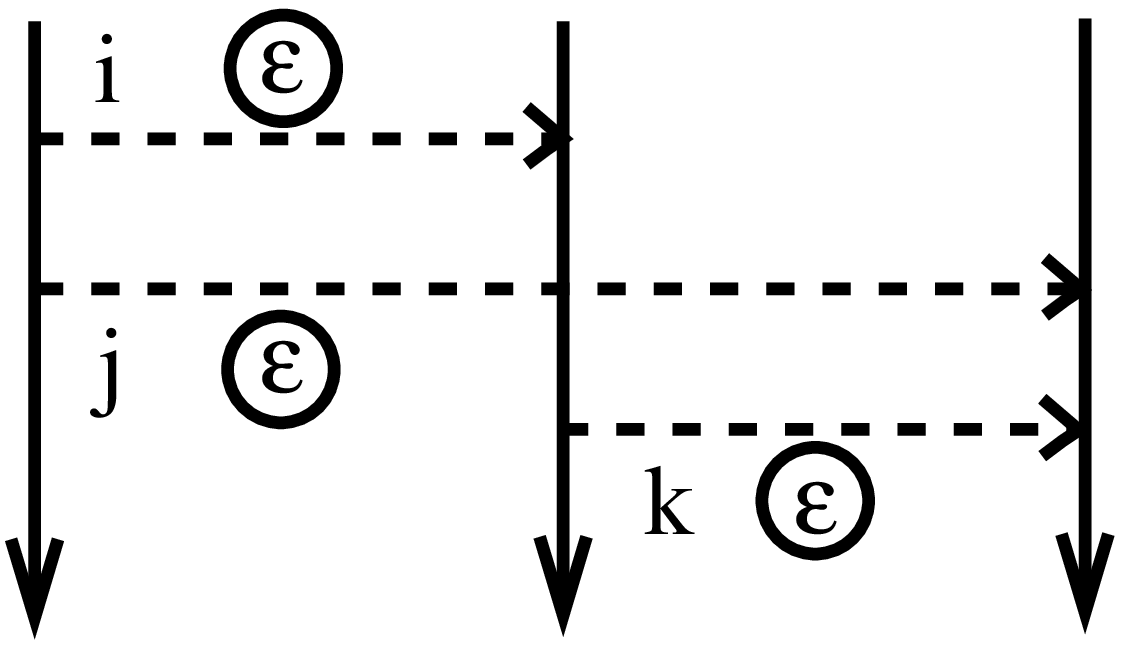}} \end{array}+\begin{array}{c}\scalebox{.23}{\psfig{figure=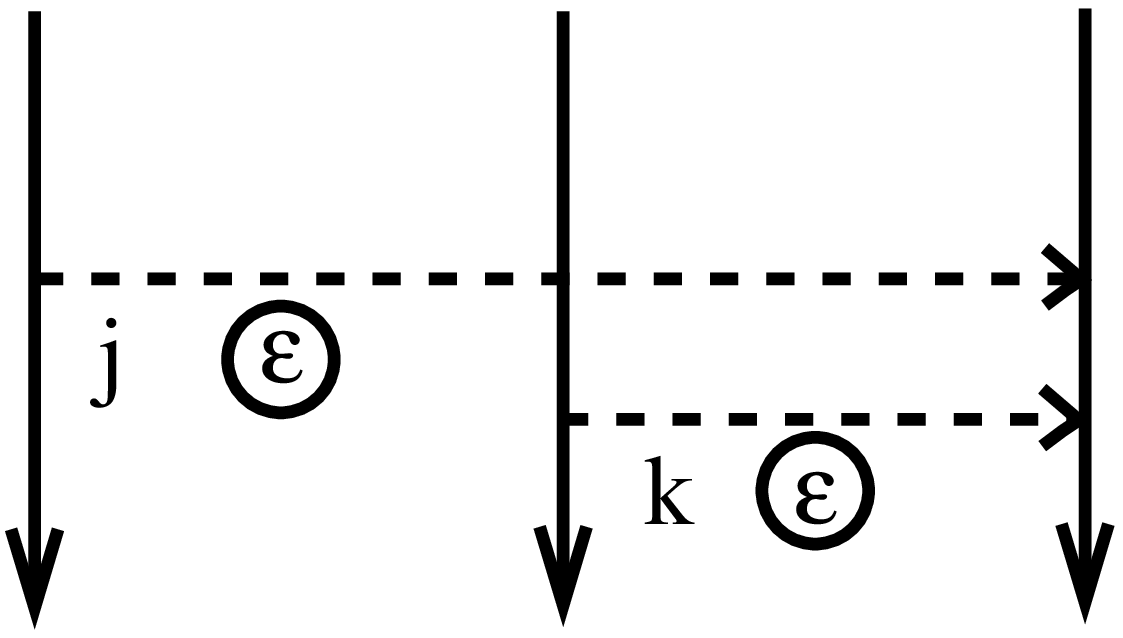}} \end{array}+\begin{array}{c}\scalebox{.23}{\psfig{figure=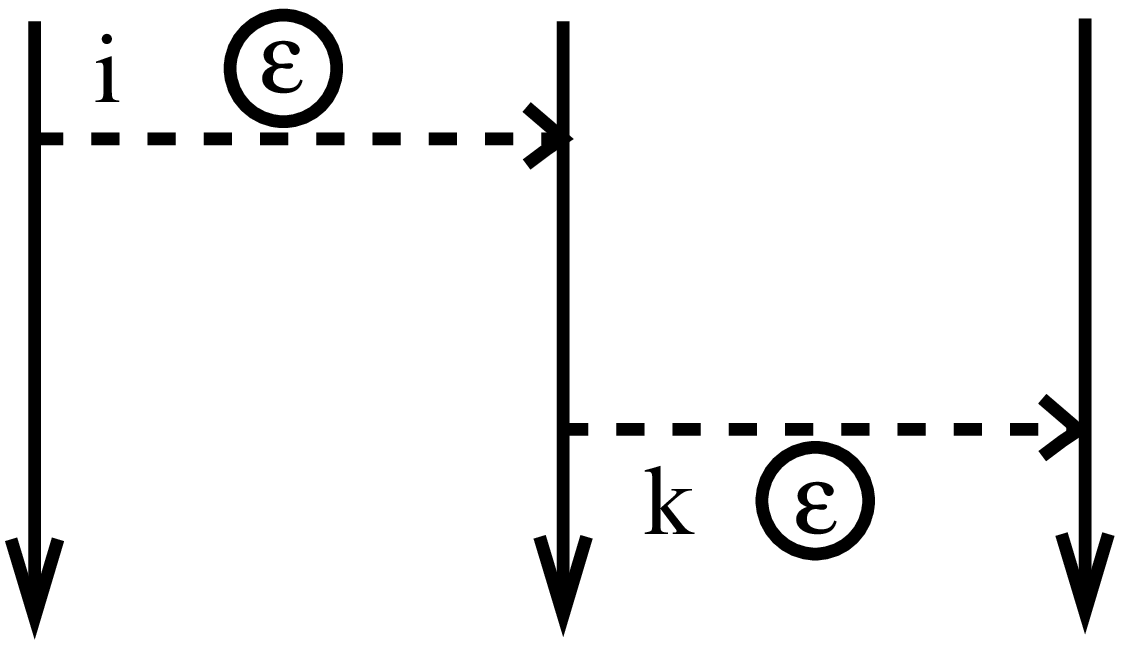}} \end{array}+\begin{array}{c}\scalebox{.23}{\psfig{figure=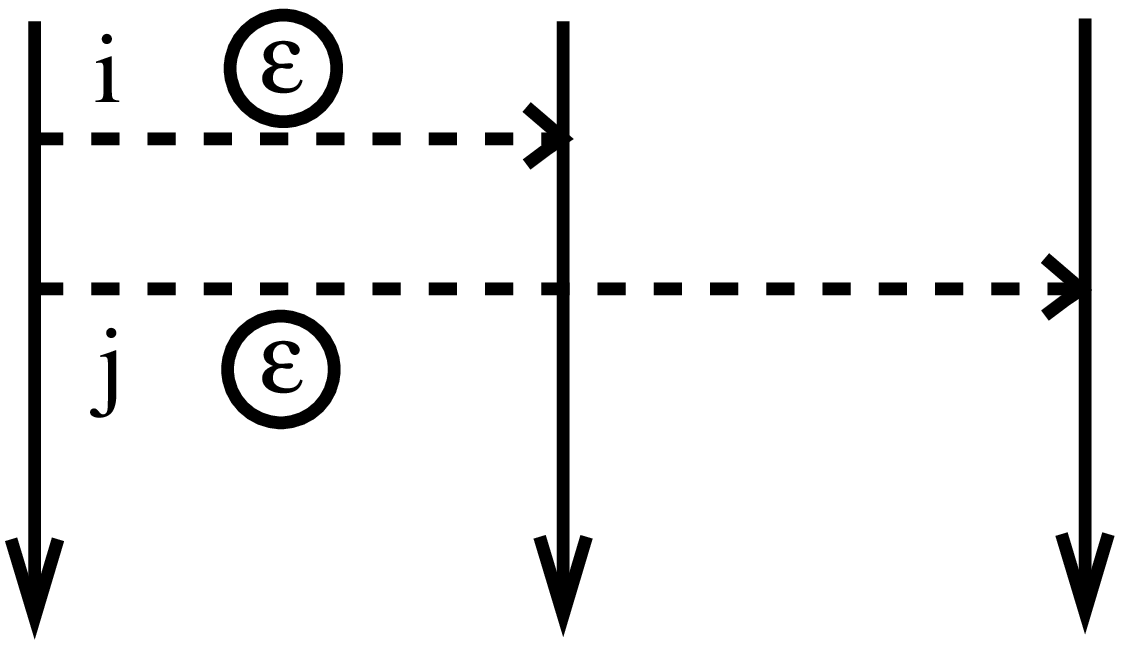}} \end{array} &=& \\ \begin{array}{c}\scalebox{.23}{\psfig{figure=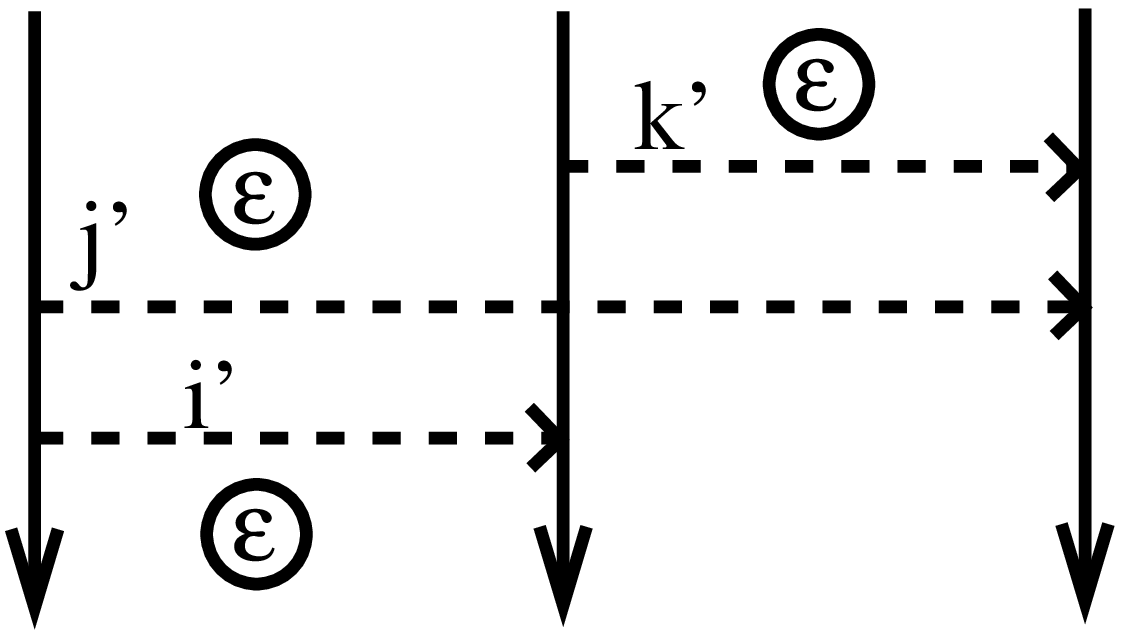}} \end{array}+\begin{array}{c}\scalebox{.23}{\psfig{figure=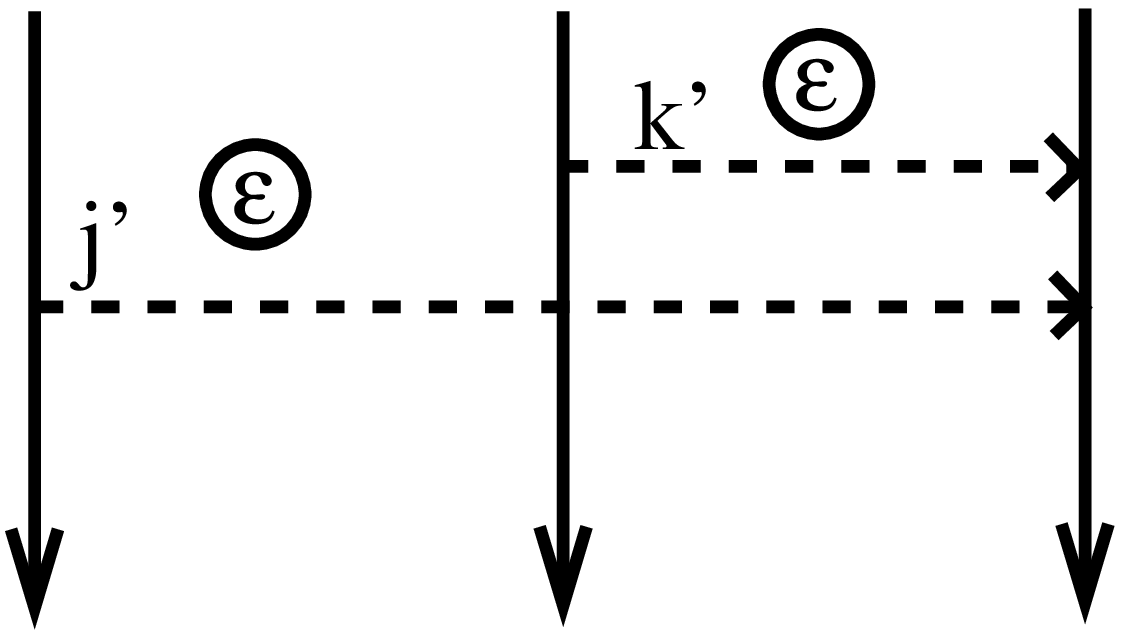}} \end{array}+\begin{array}{c}\scalebox{.23}{\psfig{figure=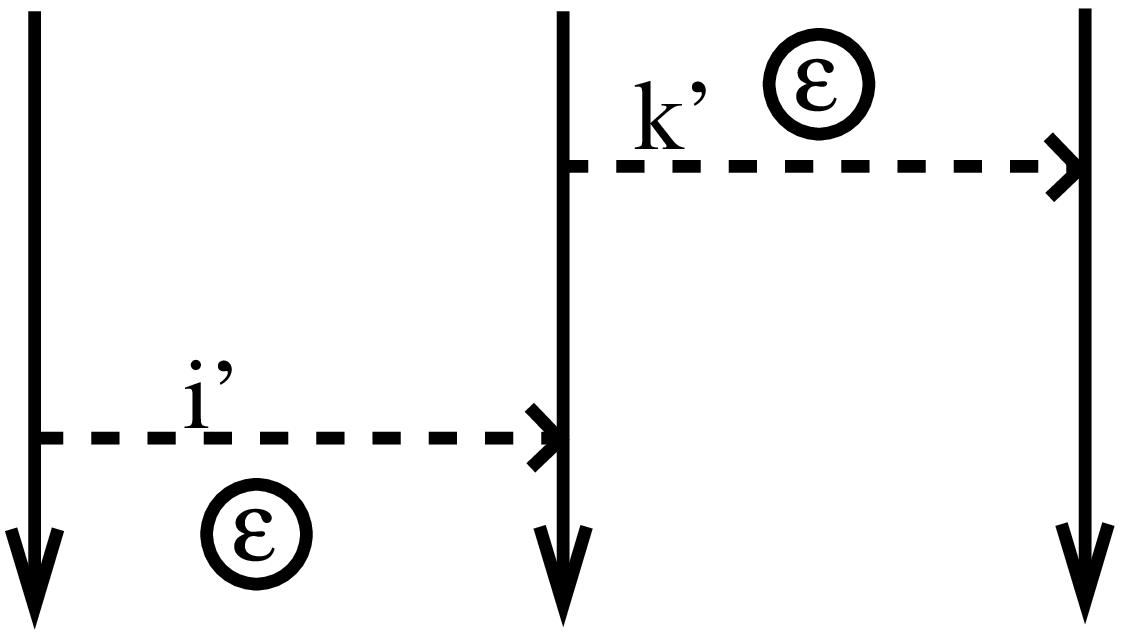}} \end{array}+\begin{array}{c}\scalebox{.23}{\psfig{figure=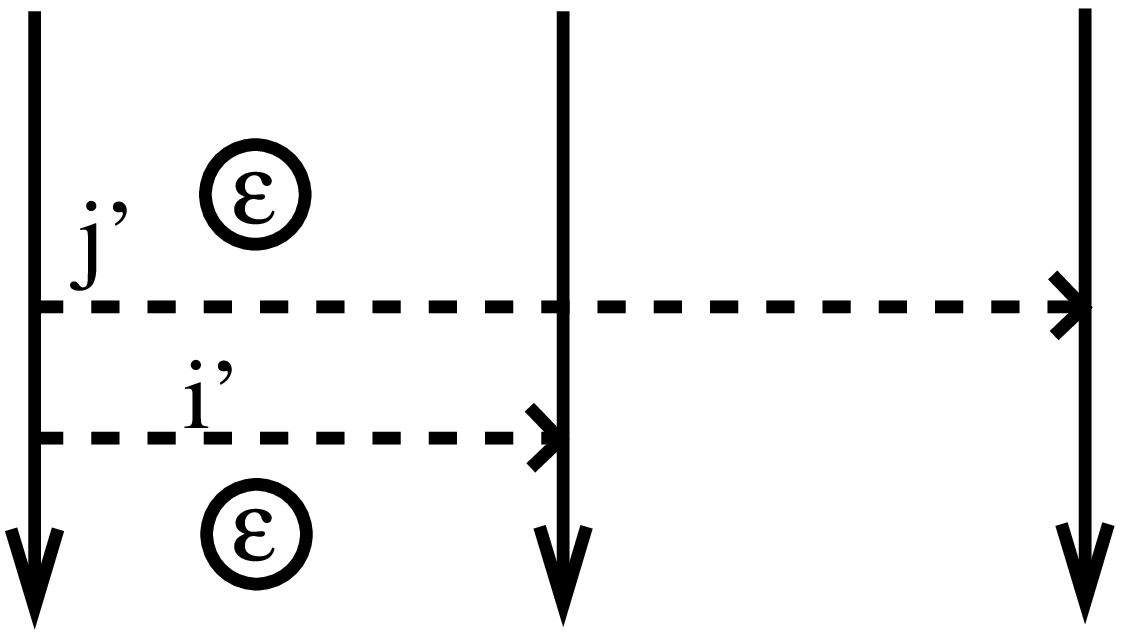}} \end{array} & & \\
\end{eqnarray*}
In $\vec{\text{Q2}}[m]$, we only have arrows with labels up to $m+1$.  In $\vec{\text{Q3}}[m]$, we include all possibilities where $i=i'$, $j=j'$, $k=k'$, and either $i>j=k$, $j>i=k$, $k>i=j$, or $i=j=k=m+1$.

We note that these relations generate the image of the relations in Figure \ref{gaussmoves} under the map $I[m] \circ \Lambda[m]$, subject to the properties of the iterates of $f$(compare with \cite{GPV}).

Let $\vec{A}_t[m]$ denote those diagrams having more than $t$ arrows and all labels $\le m+1$. We define:
\[
\vec{\mathscr{X}}[m]=\frac{\mathbb{Z}[\vec{\mathscr{A}}[m]]}{\left<\vec{\text{Q1}}[m],\vec{\text{Q2}}[m], \vec{\text{Q3}}[m]\right>},\,\,\,
\vec{\mathscr{X}}_t[m]=\frac{\mathbb{Z}[\vec{\mathscr{A}}[m]]}{\left<\vec{\text{Q1}}[m],\vec{\text{Q2}}[m], \vec{\text{Q3}}[m], \vec{A}_t[m] \right>}.
\]
For Gauss diagrams on $\mathbb{R}$, the rational vector space $\mathbb{Q} \otimes_{\mathbb{Z}} \vec{\mathscr{X}}[m]$ has the structure of an algebra.  The multiplication is given by a  map $\vec{\mathscr{A}}[m] \times \vec{\mathscr{A}}[m] \to \vec{\mathscr{A}}[m]$ which is defined by $(D_1,D_2) \to D_1D_2$, i.e. the simple concatenation of the arrow diagrams. This is the same multiplication map that we have for the Polyak algebra \cite{GPV}. We will not use the structure of the algebra, but this is what we mean by ``algebra'' in the term ``$f^m$-labelled Polyak algebra''.

The fact that that the $f^m$-labelled Polyak algebra gives rise to invariants of Kauffman finite-type follows from the definitions and arguments which are available in the literature.  For example, the $m=1$ case was considered in \cite{CM}. We record the result as a Lemma below.

\begin{lemma}[Property \ref{prop1}] \label{kauffft} Let $P$ be a parity of long flat virtual knots. If $v \in \text{Hom}_{\mathbb{Z}}(\vec{\mathscr{X}}_t[m], \mathbb{Q})$, then $v \circ I[m] \circ \Lambda[m]:\mathbb{Z}[\mathscr{D}] \to \mathbb{Q}$ is a Kauffman finite-type invariant of order $\le t$.
\end{lemma}

Next we prove that the lattice is surjective and commutative.

\begin{lemma}[Horizontal Surjectivity of the Lattice] \label{horsurj} Let $\pi_t[m]:\vec{\mathscr{X}}_t[m] \to \vec{\mathscr{X}}_{t-1}[m]$ denote the natural map of the quotient spaces. The following sequence is exact for all $m \in \mathbb{Z}$:
\[
\xymatrix{\vec{\mathscr{X}}_t[m] \ar[r] & \vec{\mathscr{X}}_{t-1}[m] \ar[r] & 0}.
\]
\end{lemma}
\begin{proof} This follows from the fact that $\vec{A}_t[m] \subseteq \vec{A}_{t-1}[m]$.
\end{proof}

We now describe the vertical maps in the lattice.  We define maps for $1 \le n \le m \le \infty$ as follows:
\[
d[m \to n]: \mathbb{Z}[\vec{\mathscr{A}}[m]] \to \mathbb{Z}[\vec{\mathscr{A}}[n]].
\]
If $D \in \vec{\mathscr{A}}[m]$, relabel any arrow of $D$ having label $k>n+1$ by $n+1$. The resulting diagram is $d[m \to n](D)$. We note that by this definition, $d[m \to m]=1$ (the identity map).

\begin{lemma}[Vertical Surjectivity of the Lattice] \label{dmnlemm} For any $t \in \mathbb{N}$, $d[m \to n]: \mathbb{Z}[\vec{\mathscr{A}}[m]] \to \mathbb{Z}[\vec{\mathscr{A}}[n]]$ descends to a map of the quotients $d_t[m\to n]:\vec{\mathscr{X}}_t[m] \to \vec{\mathscr{X}}_t[n]$. Moreover, the following sequence is exact:
\[
\xymatrix{
\vec{\mathscr{X}}_t[m] \ar[r]^{d_t[m \to n]} & \vec{\mathscr{X}}_t[n] \ar[r] & 0
}.
\]
\end{lemma}

\begin{proof} Since $d[m\to n]$ preserves the number of arrows, it is only necessary to check the relations $\vec{\text{Q1}}[m]$, $\vec{\text{Q2}}[m]$, and $\vec{\text{Q3}}[m]$.  In a $\vec{\text{Q1}}[m]$ relation, the isolated arrow having label $m+1 \ge n+1$ gets relabelled with an $n+1$.  Therefore, the image of a $\vec{\text{Q1}}[m]$ relation is a $\vec{\text{Q1}}[n]$ relation.  For a $\vec{\text{Q2}}[m]$ relation, the labels of the affected arrows are the same and hence will be the same after the application of $d[m \to n]$.

The $\vec{\text{Q3}}[m]$ relation has several cases.  If $x,y,z \ge n+1$ or $x,y,z <n+1$, the result is trivially true. Suppose then that exactly one label, say $x$, is $\ge n+1$.  Then it must be that $y=z<n+1$.  Hence, the labels $x',y',z'$ in the image of $d[m \to n]$ will satisfy $y'=z'<n+1=x'$.  Therefore, the image of any $\vec{\text{Q3}}[m]$ relation is a $\vec{\text{Q3}}[n]$ relation.  

Since $d[m \to n]$ is a surjection, it follows that the sequence is exact.  
\end{proof}

\begin{theorem} \label{xinfty} The following sequence is exact.  Hence, $\text{Hom}_{\mathbb{Z}}(\vec{\mathscr{X}}_t[n],\mathbb{Q})$ may be identified as a subgroup of $\text{Hom}_{\mathbb{Z}}(\vec{\mathscr{X}}_t[m],\mathbb{Q})$ for all $m \ge n$.
\[
\xymatrix{
0 \ar[r] & \text{Hom}_{\mathbb{Z}}(\vec{\mathscr{X}}_t[n],\mathbb{Q}) \ar[r]^{(d_t[m \to n])^*} &  \text{Hom}_{\mathbb{Z}}(\vec{\mathscr{X}}_t[m],\mathbb{Q})
}
\]
\end{theorem}
\begin{proof} This follows immediately from Lemma \ref{dmnlemm}.
\end{proof}

\begin{theorem}[Commutativity of the Lattice] The following diagram commutes for all $t$, $m$, $n$, with $m \ge n$.
\[
\xymatrix{
\vec{\mathscr{X}}_t[m] \ar[r] \ar[d] & \vec{\mathscr{X}}_{t-1}[m] \ar[d] \\
\vec{\mathscr{X}}_t[n] \ar[r] & \vec{\mathscr{X}}_{t-1}[n] \\
}
\]
\end{theorem}
\begin{proof} This is clear from the definitions of the maps $\pi_t[m]$ and $d[m \to n]$ and the proofs of Lemmas \ref{horsurj} and \ref{dmnlemm}.
\end{proof}

Lastly we need to show that each row is an extension of the Polyak sequence of groups $\vec{\mathscr{P}}_t$.  Let $\vec{\text{V1}}[m]$, $\vec{\text{V2}}[m]$, $\vec{\text{V3}}[m]$ denote those $\vec{\text{Q1}}[m]$,$\vec{\text{Q2}}[m]$, $\vec{\text{Q3}}[m]$ relations, respectively, where all arrows are labelled $m+1$.  Let $\vec{E}[m]$ denote those dashed signed arrow diagrams where all arrows are labelled $m+1$.  Let $\vec{E}_t[m]$ denote those diagrams in $\vec{E}[m]$ which have more than $t$ arrows.  We define the quotient group $\vec{\mathscr{E}}_t[m]$ to be:
\[
\vec{\mathscr{E}}_t[m]=\frac{\mathbb{Z}[\vec{E}[m]]}{\left<\vec{E}_t[m],\vec{\text{V1}}[m], \vec{\text{V2}}[m], \vec{\text{V3}}[m] \right>}.
\]
\begin{theorem}[Property \ref{prop2}] \label{epolyakisom} For all $m$ and $t$, $\vec{\mathscr{E}}_t[m]$ is isomorphic to the Polyak group $\vec{\mathscr{P}}_t$. Moreover, the following sequence is exact for every $m$:
\[
0 \to \vec{\mathscr{E}}_t[m] \to \vec{\mathscr{X}}_t[m].
\]
Moreover, for every $v \in \text{Hom}_{\mathbb{Z}}(\vec{\mathscr{X}}_t[m], \mathbb{Q})$, the value of $v \circ I[m] \circ \Lambda[m]$ on any classical diagram is determined by the restriction of $v$ to the subgroup $\vec{\mathscr{E}}_t[m]$.
\end{theorem}
\begin{proof} The first fact is clear from the definitions (see \cite{GPV}). For the second fact, note that for any classical diagram $D$ we have $f(D)=D$.  Therefore, $\Lambda[m](D)$ labels all the arrows of $D$ with $m+1$. The conclusion follows from the definition of $\vec{\mathscr{E}}_t[m]$.
\end{proof}

\section{Property \ref{prop3}: $f^m$-labelled Conway Polynomial}\label{fmcon} In this section, we define an $f^m$-labelled Conway polynomial for every $m$. It is proved that each polynomial has a representation in the lattice.  For this restriction, we consider only the Gaussian parity and Gauss diagrams on $\mathbb{R}$.  Under these conditions, the polynomials are all distinct.  However, all of the extensions satisfy the same skein relation.

\subsection{The Classical Conway Polynomial} The Conway polynomial for classical links is uniquely determined by the following skein relation:
\[
\nabla\left(\begin{array}{c}\scalebox{.05}{\psfig{figure=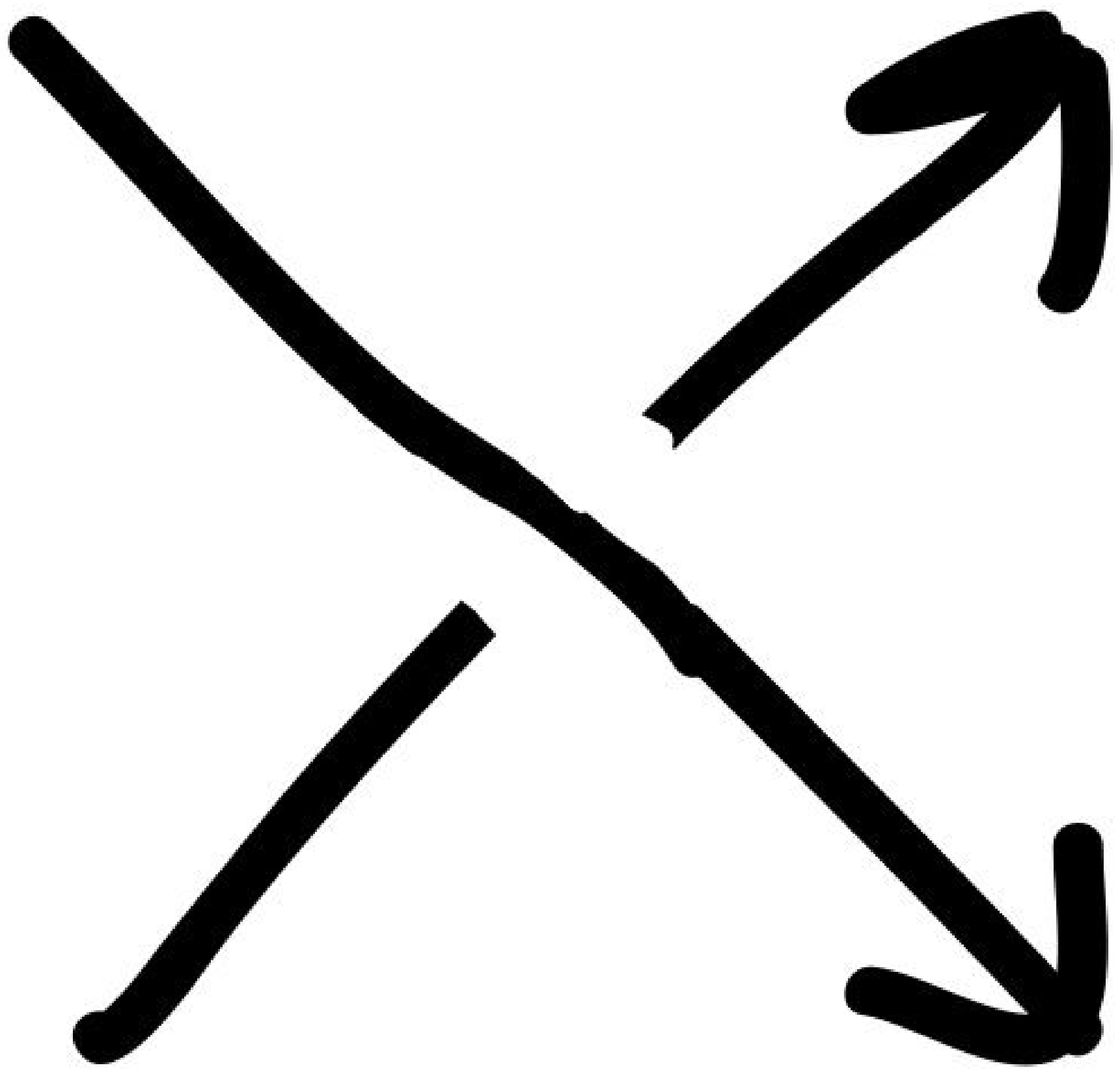}} \end{array}\right)-\nabla\left(\begin{array}{c}\scalebox{.05}{\psfig{figure=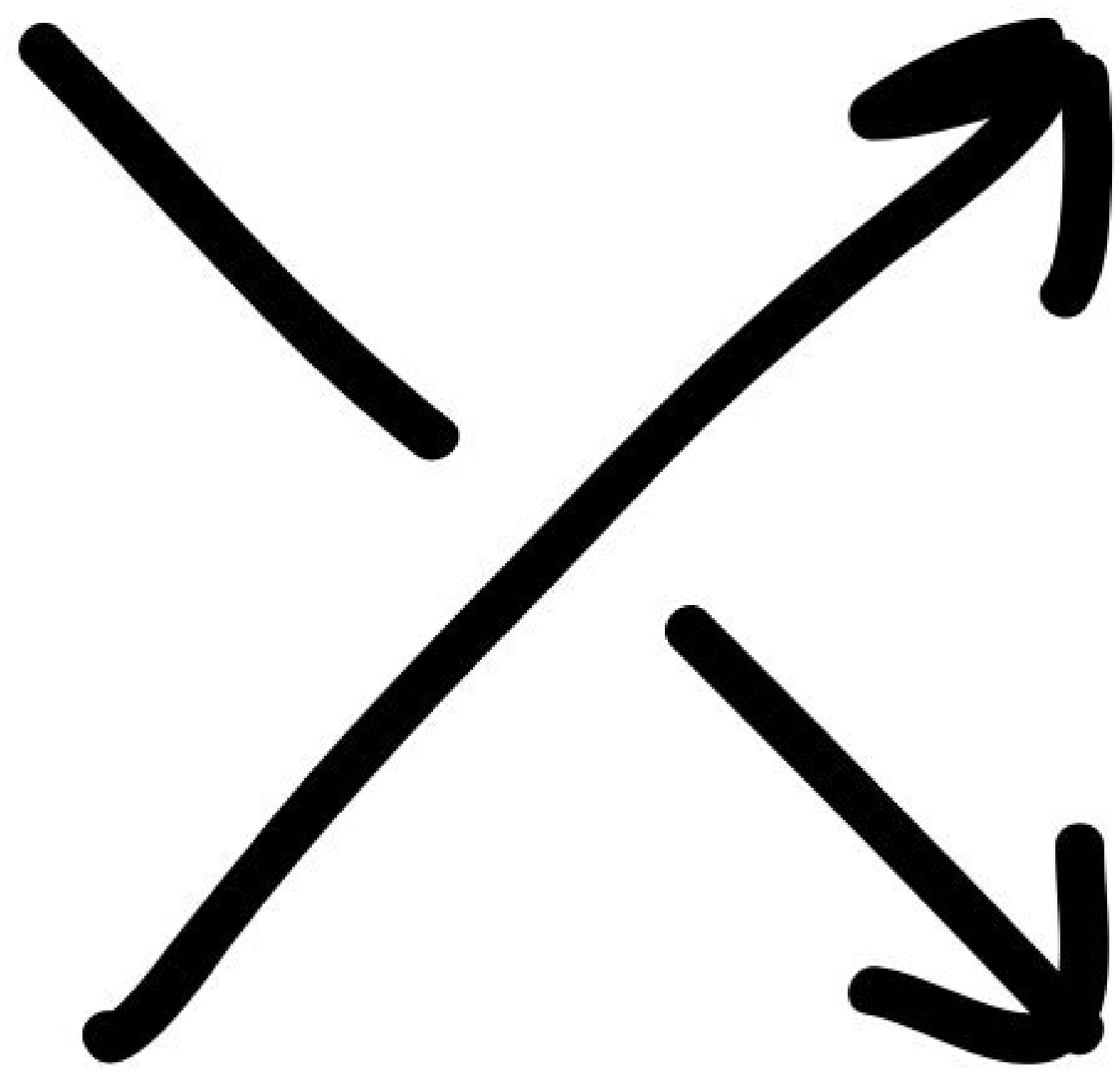}} \end{array}\right)=z \cdot \nabla\left(\begin{array}{c}\scalebox{.045}{\psfig{figure=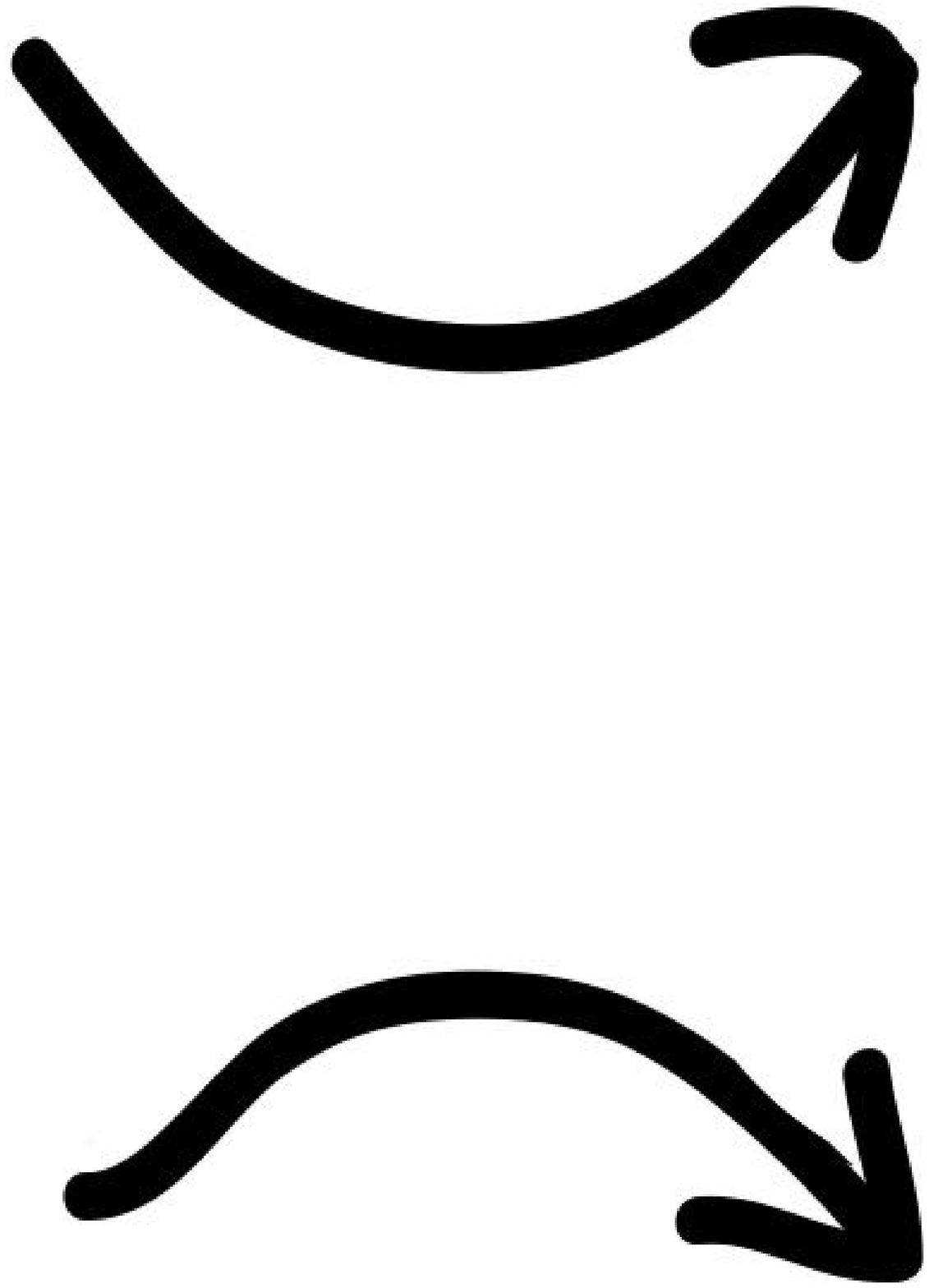}}\end{array}\right),
\]
and the condition that $\nabla(\bigcirc)=1$. The classical link diagrams $K_{\oplus}$, $K_{\ominus}$, $K_0$ form what is known as a \emph{Conway triple}.

There are many known extension of the Conway polynomial to virtual knots and virtual long knots.  Some of them satisfy a straightforward generalization of the skein relation \cite{CKR} while some do not \cite{saw}. Recently, Chmutov, Khoury, and Rossi showed that there exist two natural extensions of the Conway polynomial $\nabla _{\text{asc}}$ and $\nabla_{\text{desc}}$ to virtual long knots (see also, \cite{ChP}) which satisfy a certain skein relation. In addition, they found Gauss diagram formulae which compute the coefficients of the Conway polynomial up to any order. Related work for other knot polynomials has been done by Chmutov and Polyak \cite{ChP} and Brandenbursky and Polyak \cite{BP}. 

\begin{figure}[h]
\[
\begin{array}{c} \scalebox{.05}{\psfig{figure=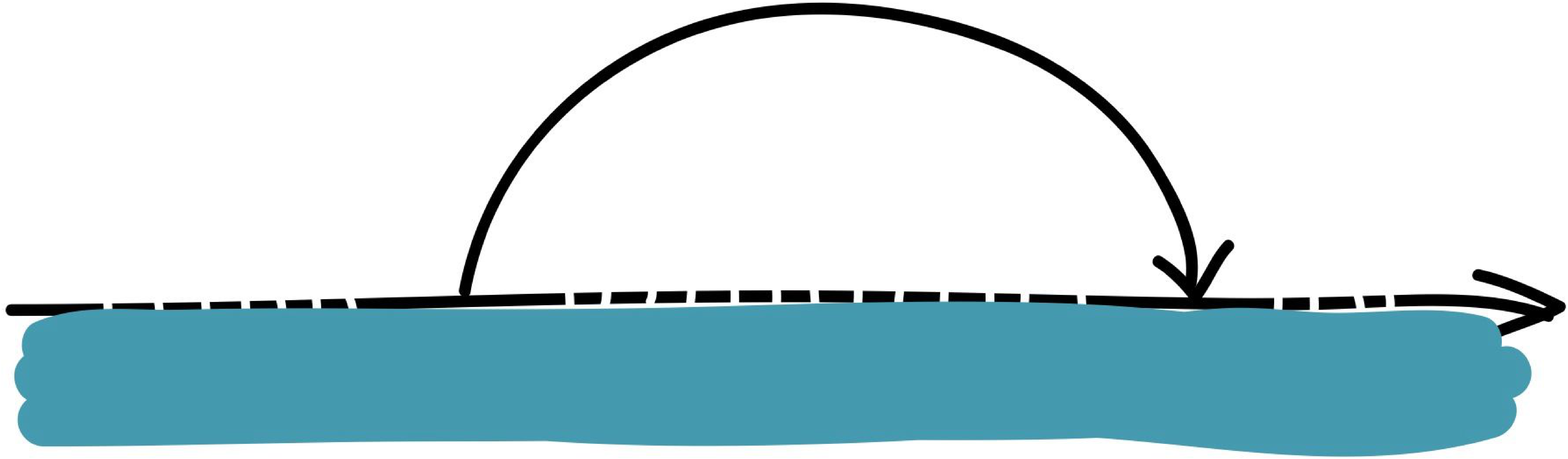}} \end{array} \to \begin{array}{c} \scalebox{.05}{\psfig{figure=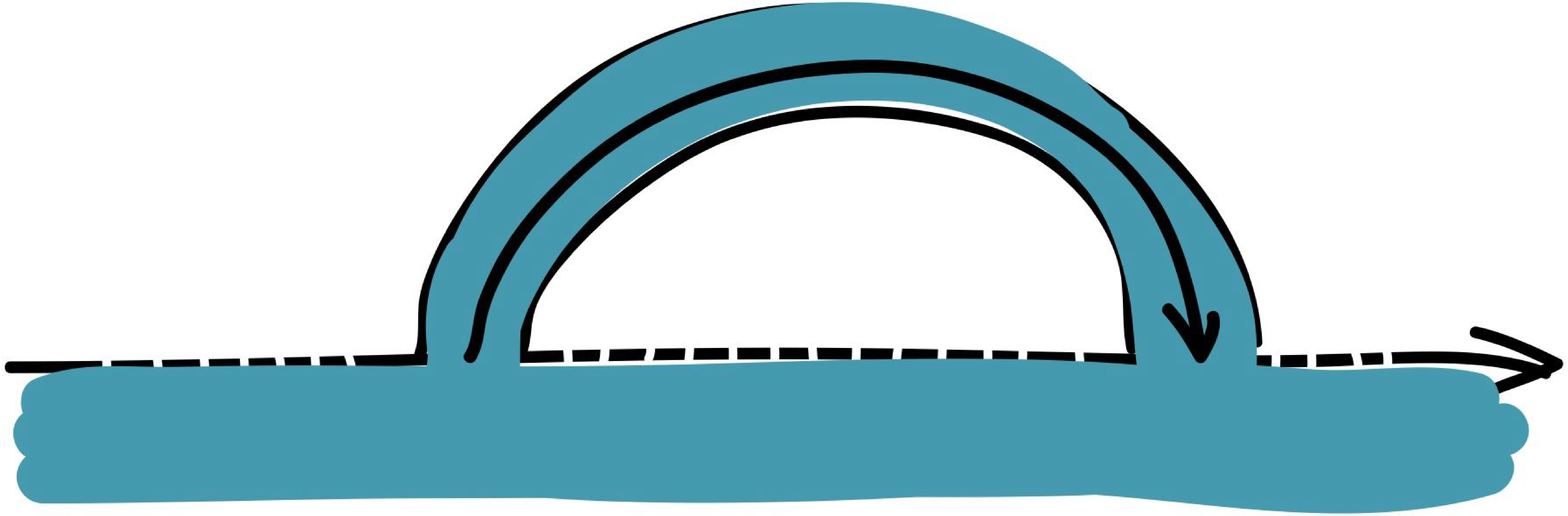}} \end{array}
\]
\caption{The oriented smoothing at an arrow.} \label{oreinsmooth}
\end{figure}

The invariant $\nabla_{\text{asc}}$ may be defined as follows\cite{CKR}. Let $D$ be a Gauss diagram on $\mathbb{R}$ and $x$ an arrow of $D$. We consider $\mathbb{R}$ to be identified with the $x$-axis in $\mathbb{R}^2$ where it bounds the lower half-plane.  The \emph{oriented smoothing} of $D$ at $x$ is obtained by gluing an untwisted band $[0,1]\times[0,1]$ to the intervals around the endpoints of $x$.  This is illustrated in Figure \ref{oreinsmooth}. If we take the oriented smoothing at every crossing of $D$, we call the resulting orientable surface the \emph{Seifert smoothing} of $D$. The diagram is said to have \emph{one component} if the number of boundary components of the Seifert smoothing is one. While traversing a one component diagram from $-\infty \to \infty$, each arrow is passed exactly twice.  If the first pass of each arrow is in the same direction of the arrow, the one component diagram $D$ is said to be \emph{ascending}.   

Let $C_{2n}$ denote the sum of all ascending Gauss diagrams on $\mathbb{R}$ having exactly $2n$ arrows, where each summand is weighted according to the product of the signs of its arrows.  It follows from \cite{CKR} and \cite{ChP} that if $c_{2n}(K)$ is the coefficient of $z^{2n}$ in $\nabla(K)$, then:
\[
c_{2n}(K)=\left<C_{2n},I(D_K)\right>, \text{ where } I(D)=\sum_{D' \subseteq D} D',  
\]
$\left<\cdot,\cdot\right>$ is the pairing $\left<D,E \right>=\delta_{DE}$, and $D_K$ is a Gauss diagram of $K$. In addition, there is an extension $\nabla_{\text{asc}}$ of the Conway to long virtual knots defined as follows:
\[
\nabla_{\text{asc}}(K)=\sum_{n=0}^{\infty} \left<C_{2n},  I(D_K)\right> z^{2n}.
\]
\textbf{Example:} \cite{CKR} There is only one ascending Gauss diagram on $\mathbb{R}$ of order $2$.  Then $c_2$ is given by:
\[
c_2(K)=\left<\begin{array}{c}\scalebox{.5}{\psfig{figure=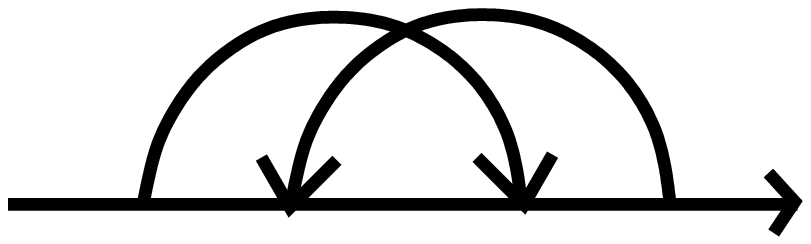}} \end{array},I(D_K)\right>.
\]
\subsection{Definition of $f^m$-Conway Polynomial} The $f^m$-labelled Conway polynomial is constructed by decomposing the combinatorial formula into its ``even'' part and ``odd part''.  The ``even part'' is killed and replaced with its ``$\infty$ part''.  

We note that parity has been used to improve a number of virtual knot polynomials.  For example, this was done in \cite{Af}.  In addition, there is the \emph{parity bracket polynomial} of V.O. Manturov. The technique presented in the present paper is somewhat different.

A combinatorial formula is a linear combination of Gauss diagrams, $F=\sum_{i=1}^N a_i F_i$, where $a_i \in \mathbb{Z}$ for $1 \le i \le N$. A combinatorial formula generates a virtual knot invariant by $\left<F,I(\cdot)\right>$, where $\left<\cdot,\cdot\right>$ and $I$ were defined in the previous section.  Let $F_i^e[m]$ be the dashed arrow diagram $F_i$ with all its arrows labelled $m+1$ (for $m=\infty$, all arrows are labelled $\infty$). We define the $f^m$-even part of $F$ to be:
\[
F^e[m]=\sum_{i=1}^N a_i F_i^e[m].
\]
The $f^m$-odd part of $F$ is defined as follows.  Let $O(D)$ be the set of diagrams in $\vec{\mathscr{A}}[m]$ whose arrows and signs are identical with $D$ and such that not all of the arrow labels are $m+1$. Then the $f^m$-odd part of $F$ is defined to be:
\[
F^o[m]=\sum_{i=1}^N \sum_{F_i^o[m] \in O(F_i)} a_i F_i^o[m].
\]
For $m=\infty$, we set $F^o[m]=0$. If $n=0$, we set $F^o[m]=0$ as well.

We say that a combinatorial formula $F=\sum F_i$ is homogeneous of order $n$ if each of the $F_i$ has exactly $n$ arrows.  For example, the formulae for the Conway coefficients, $C_{2n}$ are homogeneous of order $2n$.
 
\begin{theorem} \label{evenodddecomp} Let $m \in \mathbb{N} \cup \{\infty\}$. If $F$ is a homogeneous GPV combinatorial formula of order $n$, then $\left<F^e[m],\cdot \right> \in \text{Hom}_{\mathbb{Z}}(\vec{\mathscr{X}}_n[m],\mathbb{Q})$,  $\left<F^o[m],\cdot \right> \in \text{Hom}_{\mathbb{Z}}(\vec{\mathscr{X}}_n[m],\mathbb{Q})$, and for every long virtual knot $K$ we have:
\[
\left<F,I(D_K) \right>=\left<F^e[m],I[m]\circ \Lambda[m](D_K) \right>+\left<F^o[m],I[m]\circ \Lambda[m](D_K) \right>.
\]
\end{theorem}
\begin{proof} This follows exactly as in the case $m=1$ \cite{CM}.
\end{proof}

We define invariants $c_{2n}[m]:\mathscr{K} \to \mathbb{Z}$ as follows:
\[
c_{2n}[m](K)=\left<C^e_{2n}[\infty],I[\infty]\circ \Lambda[\infty](D_K) \right>+\left<C^o_{2n}[m],I[m]\circ \Lambda[m](D_K) \right>.
\]
Then we define the $f^m$-Conway polynomial to be:
\[
\nabla[m](K)=\sum_{n=0}^{\infty} c_{2n}[m](K) z^{2n}.
\]

\begin{theorem}[Property 3] For all $m$, the function $\nabla[m]$ is an invariant of long virtual knots.  The $c_{2n}[m]$ coefficient of $z^{2n}$ in $\nabla[m]$ is a finite-type invariant of degree $\le 2n$. Moreover, $\nabla[m]$ is represented in the lattice given in Equation $\ref{commlatt}$. If $K$ is a classical knot, then $\nabla[m](K)=\nabla(K)$.
\end{theorem}

\begin{proof} The fact that $\nabla[m]$ is an invariant which is represented in the lattice follows from Theorem \ref{evenodddecomp}. The fact that the coefficients are of Kauffman finite-type follows from Theorem \ref{kauffft}. The final claim follows from the fact that the $f^m$-label of any classical knot is $m+1$ for every $m$. Hence, the odd part $C^o_{2n}[m]$ vanishes on classical knots for every $m$ and $n$.
\end{proof}

\subsubsection{The $\nabla[m]$ are distinct} \label{earring} We prove that for all $k$, the $c_{2k}[m]$ are distinct. To do this, we define a knot diagram $K_{m,k}$, and count the number ascending one component subdiagrams of its Gauss diagram $D_{K_{m,k}}$.

Let $D_k$ denote the chord diagram on $\mathbb{R}$ whose intersection graph is the complete graph on $2k$ vertices. To each chord of $D_k$, we add an ``earring'' of width $m-1$ as follows. An \emph{earring} of width $w$ is a chord diagram whose Gauss code is:
\[
121324354\cdots (w-1)(w-2) w (w-1) w.
\] 
In particular, the intersection graph of an earring of length $w$ is a path of length $w$.  Number the chords of $D_k$ by the order of their leftmost endpoints.  We add an earring of width $m-1$ to the left of the chord $1$ in $D_k$ so that the first chord of the earring becomes the leftmost chord and the last chord of the earring is linked with chord numbered $1$ in $D_k$.  The intersection graph of this chord diagram is a coalescence of the complete graph on $k$ vertices with a path of length $m$ at an endpoint of the path.  
\begin{figure}[h]
\[
\begin{array}{ccc}
\begin{array}{c}\scalebox{.35}{\psfig{figure=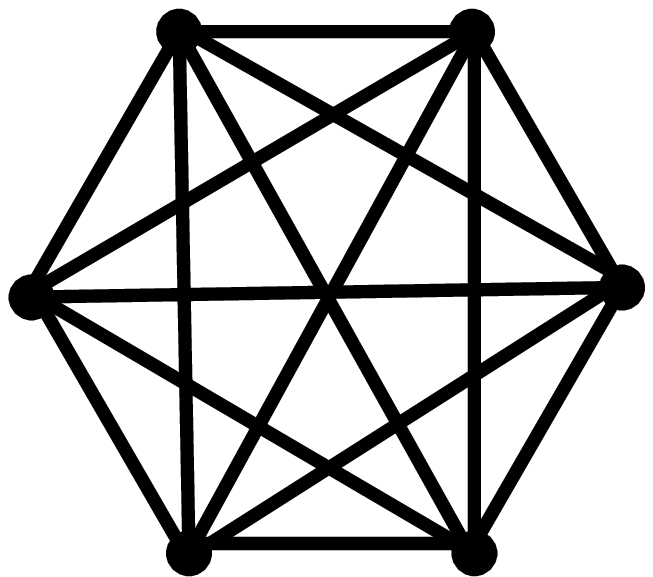}}\end{array} & \rightarrow & \begin{array}{c}\scalebox{.35}{\psfig{figure=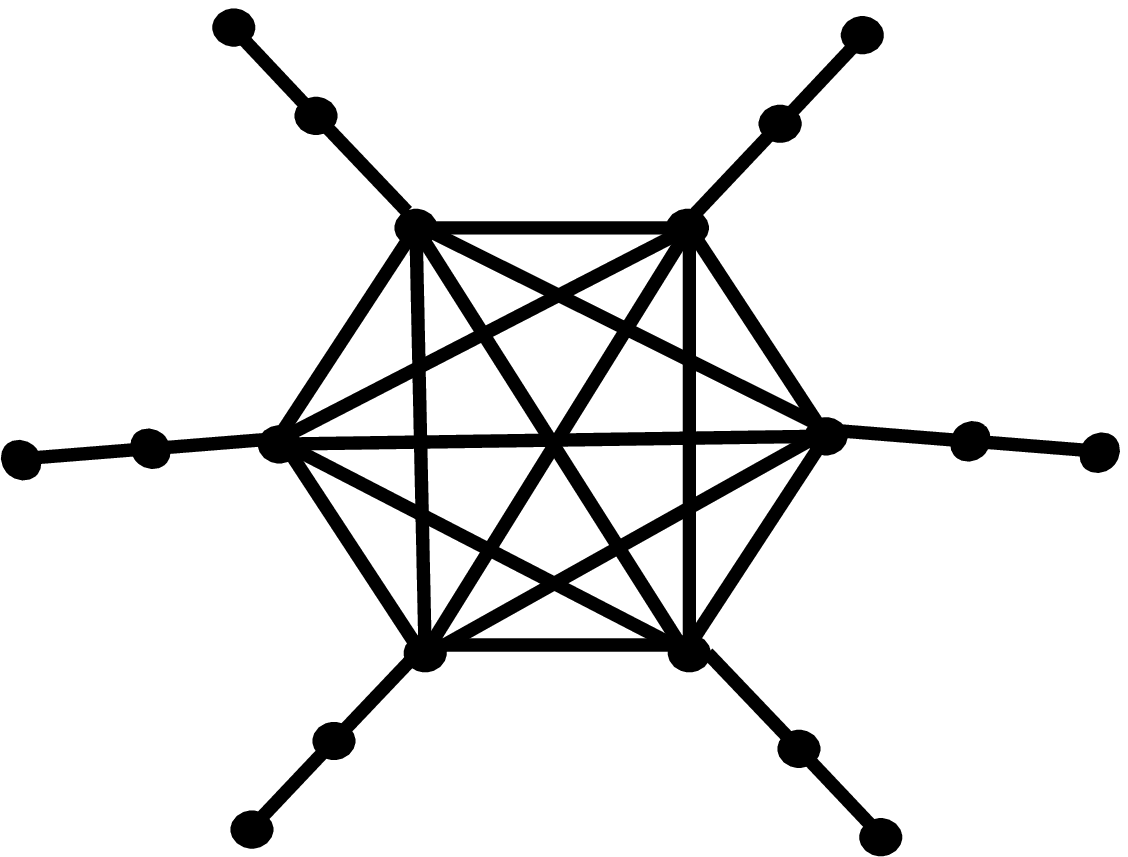}}\end{array} \\
\end{array}
\]
\caption{Adding earrings to the vertices of a complete graph.}\label{earringfig} 
\end{figure}
Similarly, we add an earring of width $m-1$ on the immediate left of the left endpoint of each odd numbered chord of $D_k$ and extending to the immediate right of the right hand endpoint of each even numbered chord of $D_k$. The resulting Gauss diagram is denoted $D_{m,k}$. Note that vertices of $D_k$ are $f^{\infty}$-labelled $m$.  Each earring has the labels $1,2,\ldots,m-1$, with the arrow labelled $m-1$ being linked with a vertex of $D_k$.
\newline
\newline
\textbf{Example:} Consider the case of $c_{2}[m]$. Then $D_{m,2}$ is given below:
\[
D_{m,2}=\begin{array}{c}\scalebox{.35}{\psfig{figure=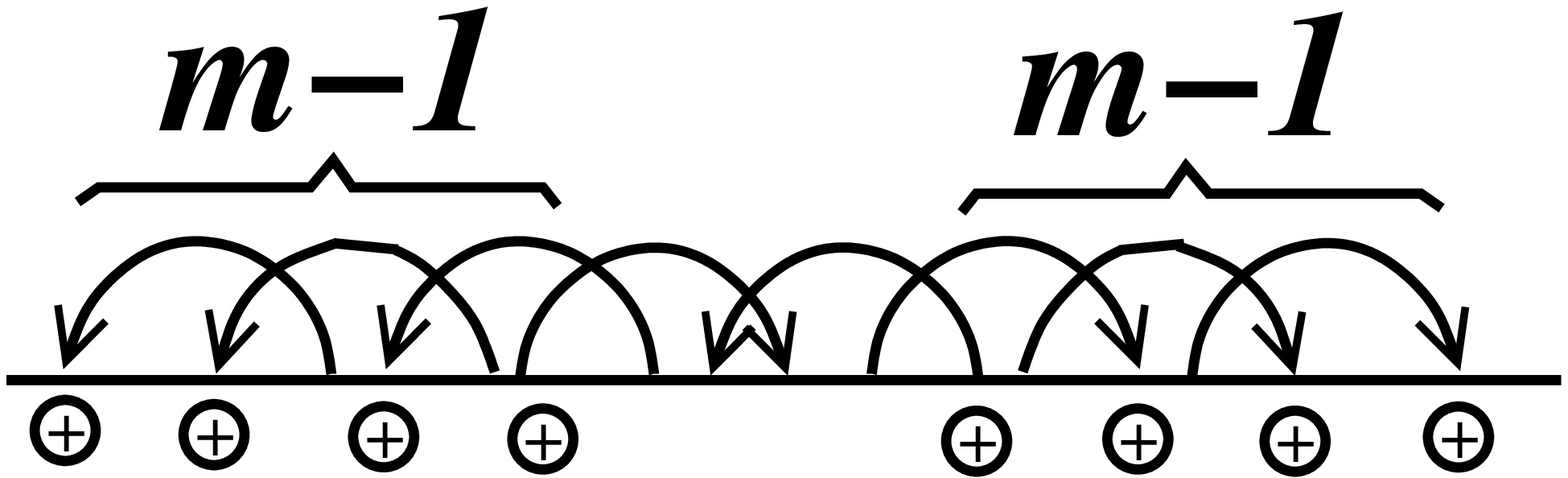}}\end{array}.
\]

Now we create a signed and directed Gauss diagram $\vec{D}_{m,k}$ from $D_{m,k}$.  First, we choose the $\oplus$ sign for each of the chords in $D_{m,k}$.  Secondly, we orient all of the odd numbered chords in $D_k$ from left to right and all of the even numbered chords of $D_k$ from right to left. We may orient the chords of the earrings arbitrarily.

\begin{lemma} \label{onecomplemm} For all $k$, the subdiagram of $\vec{D}_{m,k}$ corresponding to $D_k$ is a one component ascending diagram.
\end{lemma}
\begin{proof} By Zulli's Theorem \cite{MR1341816}, the number of boundary components is one more than the nullity of the $\mathbb{Z}_2$ adjacency matrix of $K_{2k}$ (i.e. the complete graph on two vertices).  The adjacency matrix is the $2k \times 2k$ matrix: 
\[
\left[
\begin{array}{ccccc} 
0 & 1 & 1 & \cdots & 1 \\
1 & 0 & 1 & \cdots & 1 \\
1 & 1 & 0 & \cdots & 1  \\
\vdots & \ddots & \ddots & \cdots & \vdots \\
1 & 1 & 1 & \cdots & 0 \\
\end{array}
 \right].
\]
It can be shown by induction that this matrix has nullity 0 (see also \cite{C5}). Hence, $D_k$ is of one component.  It is ascending because in the left-to-right ordering of the chords, the odd chords with respect to this ordering point right and the even chords with respect to this ordering point left.  
\end{proof}
  
\begin{lemma} For all $n$, $1 \le n <m$, we have:
\[
c_{2k}[m](\vec{D}_{m,k})-c_{2k}[n](\vec{D}_{m,k})\ge 1.
\]
\end{lemma}
\begin{proof} First note that the $f^n$-labels of $\vec{D}_{m,k}$ may be obtained from the $f^m$-labels by erasing all of the $f^m$-labels greater than $n$ and setting them equal to $n+1$. Also note that no subdiagram of $\vec{D}_{m,k}$ is counted by $C^e_{2k}[\infty]$. It therefore suffices to show that $C^o_{2k}[m]$ counts at least one more subdiagram of $\vec{D}_{m,k}$ than $C^o_{2k}[n]$.

Any ascending subdiagram of $\vec{D}_{m,k}$ which is counted by $C^o_{2k}[n]$ is also counted by $C^o_{2k}[m]$.  This is because such a subdiagram must have at least one arrow whose $f^n$-label is less than $n+1$.  Hence, its $f^m$-label must also be less than $m+1$.

Now consider the subdiagram $D_k$ of $\vec{D}_{m,k}$ whose intersection graph is isomorphic to the complete graph on $2k$ vertices. By Lemma \ref{onecomplemm}, this subdiagram is ascending.  Recall that all of the arrows of $D_k$ are all $f^m$-labelled as $m$.  Then the $f^n$-label is $n+1$.  It follows that $D_k$ contributes $1$ to $c_{2k}[m]$ and $0$ to $c_{2k}[n]$. This completes the proof of the lemma.   
\end{proof}

It follows immediately that for all $k$, if $m_1 \ne m_2$, then there exists a long virtual knot $K$ such that $c_{2k}[m_1](K) \ne c_{2k}[m_2](K)$. Hence, $\nabla[m_1](K) \ne \nabla[m_2](K)$.

\subsubsection{The $\nabla[m]$ satisfy a skein relation} In this section, we show that the $f^m$-labelled Conway polynomials also satisfy a skein relation: 
\[
\nabla[m](K_{\oplus,\oplus})-\nabla[m](K_{\oplus,\ominus})-\nabla[m](K_{\ominus,\oplus})+\nabla[m](K_{\ominus,\ominus})=z^2 \nabla[m](K_{00}),
\]
where the affected crossings are linked, both $f^{\infty}$-labelled $\infty$, and the smoothing does not change the $f^{\infty}$-labels of the remaining crossings. This Conway quintuple is given in Figure \ref{conquint}. The relative configuration of the crossings is depicted in Figure \ref{infincrossarrows}.

\begin{figure}[h]
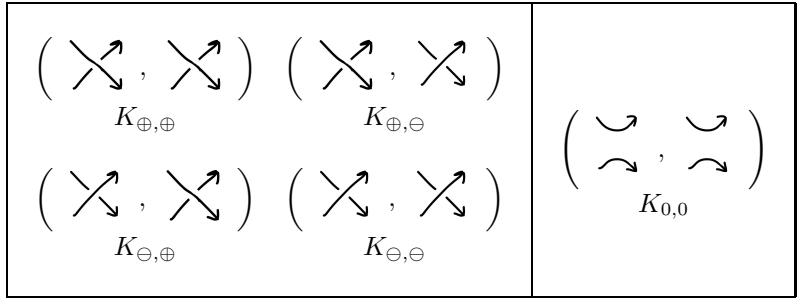

\[
\begin{array}{|c|c|} \hline & \\
\begin{array}{cc}
\left(\begin{array}{c}\scalebox{.04}{\psfig{figure=arrowdefplus.eps}} \end{array},
\begin{array}{c}\scalebox{.04}{\psfig{figure=arrowdefplus.eps}} \end{array} \right) & 
\left(\begin{array}{c}\scalebox{.04}{\psfig{figure=arrowdefplus.eps}} \end{array},\begin{array}{c}\scalebox{.04}{\psfig{figure=arrowpolydefminus.eps}} \end{array} \right) \\  K_{\oplus,\oplus} &  K_{\oplus,\ominus} \\ \\
\left(\begin{array}{c}\scalebox{.04}{\psfig{figure=arrowpolydefminus.eps}} \end{array},\begin{array}{c}\scalebox{.04}{\psfig{figure=arrowdefplus.eps}} \end{array} \right)&
\left(\begin{array}{c}\scalebox{.04}{\psfig{figure=arrowpolydefminus.eps}} \end{array},\begin{array}{c}\scalebox{.04}{\psfig{figure=arrowpolydefminus.eps}} \end{array} \right) \\   K_{\ominus,\oplus} &  K_{\ominus,\ominus} \\ & \\
\end{array} & 
\begin{array}{c}  
\left(\begin{array}{c}\scalebox{.04}{\psfig{figure=arrowpolydefsmooth.eps}} \end{array},\begin{array}{c}\scalebox{.04}{\psfig{figure=arrowpolydefsmooth.eps}} \end{array} \right) \\ K_{0,0}
\end{array} \\ \hline
\end{array}
\]
\caption{A virtual Conway quintuple.}\label{conquint}
\end{figure}
We note that when the arrows are crossed, then smoothing along both arrows gives another virtual knot i.e. the number of connected components is preserved. It is necessary to use a Conway quintuple as opposed to the traditional Conway triple for this very reason.

\begin{figure}[h]
\scalebox{.5}{\psfig{figure=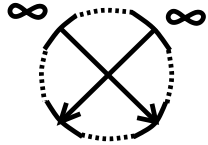}}
\caption{Configuration of arrows in the $f^m$-skein relation.} \label{infincrossarrows}
\end{figure}

\begin{lemma} Let $m$ be given.  Suppose that $(K_{\oplus,\oplus},K_{\oplus,\ominus},K_{\ominus,\oplus},K_{\ominus,\ominus}, K_{00})$ is a virtual Conway quintuple, such that the drawn arrows are linked, both $f^{\infty}$-labelled $\infty$, and such that the oriented smoothing does not change the $f^{\infty}$-labelling of the remaining arrows.  Then:
\[
c_{2k}[m](K_{\oplus,\oplus})-c_{2k}[m](K_{\oplus,\ominus})-c_{2k}[m](K_{\ominus,\oplus})+c_{2k}[m](K_{\ominus,\ominus})=c_{2k-2}[m](K_{00}).
\]
\end{lemma}
\begin{proof} The proof is similar to the proof of the skein relation in \cite{CKR}.  We set up a one-to-one correspondence of ascending diagrams involved in the relation and show that they are counted with the same weight.

Consider $f^m$-labelled subdiagrams of $K_{\varepsilon_1,\varepsilon_2}$ where $\varepsilon_1 \cdot \varepsilon_2 \ne 0$. Note that all corresponding arrows for the LHS diagrams have the same $f^m$-label. A subdiagram may contain $0$, $1$, or $2$ of the drawn arrows. Diagrams with zero arrows have no contribution on LHS. Every subdiagram which has one of the drawn arrows occurs also as a subdiagram of some $K_{\varepsilon_3,\varepsilon_4}$ such that the coefficient $c_{2k}[m](K_{\varepsilon_3,\varepsilon_4})$ is $-1$ times the coefficient of $c_{2k}[m](K_{\varepsilon_1,\varepsilon_2})$. Hence, there is no contribution.   

Hence we must show that there is a one-to-one correspondence between those ascending subdiagrams on LHS which contain both of the drawn arrows and the ascending subdiagrams of $K_{00}$. Moreover, the correspondence must preserve the weight of each diagram. It is sufficient to show that if $D_{00}$ is an ascending subdiagram of $K_{00}$ which is counted by $C^e_{2k}[\infty]$ or $C^o_{2k}[m]$, then there is exactly one $\varepsilon_1$ and $\varepsilon_2$ such that $D_{\varepsilon_1, \varepsilon_2}$ is an ascending subdiagram of $K_{\varepsilon_1,\varepsilon_2}$ and such that $D_{00}$ is counted by the corresponding formula $C^e_{2k}[\infty]$ or $C^o_{2k}[m]$ with the same weight.

Fix a pair of signs $\varepsilon_1,\varepsilon_2$ and an ascending subdiagram $D$ of $K_{\varepsilon_1,\varepsilon_2}$ which contains both of the drawn arrows. Let $D_{00}$ denote the Gauss diagram obtained from $D$ by smoothing along the two arrows. We may consider this as a subdiagram of $K_{00}$. It is easy to see that $D_{00}$ is still ascending. If all the $f^{\infty}$-labels of $D$ are $\infty$, then $C^e_{2k}[\infty]$ counts $D$ with a weight of $\varepsilon_1 \cdot \varepsilon_2 \cdot \sigma$ while the $C_{2k-2}^{e}[\infty]$ counts $D_{00}$ with a weight of $\sigma$.  We see that the total contribution to LHS and RHS is the same. If every arrow of $D$ is $f^{\infty}$-labelled $m+1$ or greater then there is no contribution on LHS or RHS of the equation. Finally suppose that there is an arrow of $D$ having $f^{\infty}$-label less than $m+1$. Then by hypothesis, $D_{00}$ (considered again as a subdiagram of $K_{00}$) also has an arrow with $f^m$-label less than $m+1$).  Hence, there is an equal contribution on LHS from $C^o_{2k}[m]$ and on RHS from $C^o_{2k-2}[m]$.

Now we consider subdiagrams of $K_{00}$. First note that the virtual knot diagram $K_{00}$ specifies a pair of arcs $A_1,A_3$ from the smoothing of the first crossing and a pair of arcs $A_2,A_4$ from the smoothing of the second crossing.  This can be used to draw a homeomorphic copy of $\mathbb{R}$ as in Figure \ref{nonstanr}. We may draw on this copy of $\mathbb{R}$ a Gauss diagram of $K_{00}$ (or any of its subdiagrams) that preserves the order of passing of the arcs $A_i$. 

Let $D_{00}$ be an ascending one component subdiagram of $K_{00}$. Delete the arcs $A_i$ and fill in the intervals $\alpha_j$ on $\mathbb{R}$. Draw two chords $x$ and $y$ with endpoints in $\alpha_a$, $\alpha_c$ and $\alpha_b,\alpha_d$. Orient these chords so that the diagram is ascending. As all directions of these two chords are included on LHS of the skein relation, these directions specify a unique choice of signs $\varepsilon_1, \varepsilon_2$. We denote the constructed subdiagram of $K_{\varepsilon_1,\varepsilon_2}$ as $D_{\varepsilon_1,\varepsilon_2}$.

Note also that the $f^m$-labels of the arrows $x$ and $y$ are $m+1$ in all of $K_{\oplus,\oplus}$, $K_{\oplus,\ominus}$, $K_{\ominus,\oplus}$, and $K_{\ominus,\ominus}$. Hence we may give the $f^m$-label of the arrows $x$ and $y$ in $D_{\varepsilon_1,\varepsilon_2}$ to be $m+1$. Moreover, the $f^{\infty}$-labels of $x$ and $y$ are both $\infty$.

Now, if the product of the signs in $D_{00}$ is $\sigma$, then the product of the signs in $D_{\varepsilon_1,\varepsilon_2}$ is $\varepsilon_1 \cdot \varepsilon_2 \cdot \sigma$. Hence the weights of the $D_{00}$ and $D_{\varepsilon_1,\varepsilon_2}$ will be the same whenever they are both counted.  

\end{proof}
\begin{figure}[h]
\scalebox{.11}{\psfig{figure=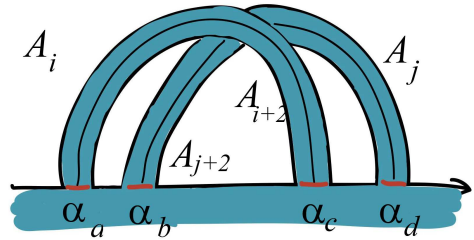}}
\caption{A non-standard copy of $\mathbb{R}$.} \label{nonstanr}
\end{figure}

\begin{theorem} The $f^m$-labelled Conway polynomials satisfy a skein relation: 
\[
\nabla[m](K_{\oplus,\oplus})-\nabla[m](K_{\oplus,\ominus})-\nabla[m](K_{\ominus,\oplus})+\nabla[m](K_{\ominus,\ominus})=z^2 \nabla[m](K_{00}),
\]
where the affected crossings are linked, both $f^{\infty}$-labelled $\infty$, and the smoothing does not change the $f^{\infty}$-labels of the remaining crossings.
\end{theorem}
\begin{proof} We compare the coefficients of $z^{2n}$ on the left hand side of the equation with the coefficient of $z^{2n-2}$ on the right hand side of the equation.  By the previous lemma, the weighted sum of the coefficients on the left must be the same as the coefficient on the right. This completes the proof.
\end{proof}

\section{Property \ref{prop4}: Infinite Dimensionality of Lattice Columns} \label{prop4sec} Let $P$ be the Gaussian parity. First we will give a set of simple finite-type invariants. Then we will show that they have an infinite dimensional subspace. Finally, we will use discrete calculus to show that these finite-type invariants are represented by combinatorial formulae in the lattice.

\label{fmwords} For $D \in \vec{\mathscr{A}}[m]$, let  $\theta[m]\left(D, \uparrow_{\oplus}^k \right)$ denote the number of arrows of $D$ signed $\oplus$ and labelled $k$.  Similarly, let $\theta[m]\left(D,\uparrow_{\ominus}^k \right)$ denote the number of arrows of $D$ signed $\ominus$ and labelled $k$. For $1 \le k \le m$, we define a function $\theta[m|k]:\mathbb{Z}[\vec{\mathscr{A}}[m]] \to \mathbb{Z}$ on generators by:
\[
\theta[m|k](D)=\theta[m]\left(D, \uparrow_{\oplus}^k \right)-\theta[m]\left(D, \uparrow_{\ominus}^k \right).
\]
Note that the definition of $\theta[m|k]$ is the same for Gauss diagrams on $\mathbb{R}$ and Gauss diagrams on $S^1$.

\begin{lemma} \label{thetaok} The function $\theta[m|k] \circ \Lambda[m]: \mathbb{Z}[\mathscr{D}] \to \mathbb{Z}$ is an invariant of virtual knots/long virtual knots.
\end{lemma}
\begin{proof} This follows from definition of the functorial map $f$.
\end{proof}

For $t \in \mathbb{N}$, define $\theta_t[m|k]:\vec{\mathscr{A}}[m] \to \mathbb{Z}$ by:
\[
\theta_t[m|k](D)=\underbrace{\theta[m|k](D) \cdot \theta[m|k](D) \cdot \ldots \cdot \theta[m|k](D) }_{t \text{ times}}.
\]
In addition, we have the map $\theta_{t_1,\ldots,t_s}[m|k_1,\ldots, k_s]:\vec{\mathscr{A}}[m] \to \mathbb{Z}$ which is defined by:
\[
\theta_{t_1,\ldots,t_s}[m|k_1,\ldots, k_s](D)=\prod_{j=1}^s \theta_{t_j}[m|k_j](D).
\]
For simplicity, we will always assume that the set $\{k_1,\ldots, k_s \}$ of labels has exactly $s$ elements (i.e. all the $k_i$ are distinct).

\begin{theorem} Let $P$ be the Gaussian parity. Let $n \in \mathbb{N}$.  Let $\mathscr{S}_n=\{\theta_n[m|m]:m \in \mathbb{N}\}$.  Then $\mathscr{S}_n$ is a rationally linearly independent set of Kauffman finite-type invariants of degree exactly $n$.  Hence, the set of Kauffman finite-type invariants of degree exactly $n$ is infinite dimensional for every $n$.  
\end{theorem}
\begin{proof} It is easy to see that $\theta_1[m|m]$ is a Kauffman finite-type invariant of degree $1$. To prove $\theta_n[m|m]$ is of degree exactly $n$, one can use a twist sequence argument (see \cite{C1}).

We will first show that for every $m$, there is a virtual knot $L_m$ such that $\theta_n[m|m](L_m) \ne 0$ and if $m_1<m$, then $\theta_n[m_1|m_1]=0$. Indeed, we take $L_m$ to be $D_{m,2}$ (see Example, Section \ref{earring}) except that the $m-1$ most rightward arrows are signed $\ominus$. Also note that the arrow directions are irrelevant. We compute:
\[
\theta_n[m_1|m_1](L_m)=\left\{\begin{array}{cc} 0 & m_1 \neq m \\ 2^n  & m_1 = m \end{array} \right. .
\]
Now suppose that there are $m_1,\ldots, m_k$, $m_i \le m_j$ for $i \le j$, and $\alpha_1 \ldots , \alpha_k \in \mathbb{Q}$ such that:
\[
\alpha_1 \theta_n[m_1|m_1]+\ldots+\alpha_k \theta_n[m_k|m_k]=0.
\]
For each $i$, $1 \le i \le k$, evaluate both sides of this equation at $L_{m_i}$. It follows that $\alpha_i=0$. Hence, $\mathscr{S}_n$ is linearly independent over $\mathbb{Q}$.   
\end{proof}

\begin{theorem} \label{thetacomb} The invariant $\theta_{t_1,\ldots, t_s}[m|k_1,\ldots,k_s]$ is represented by a combinatorial formula in the lattice:
\[
\left<F_{t_1,t_2,\ldots,t_s}[m|k_1,k_2,\ldots k_s], \cdot \right> \in \text{Hom}_{\mathbb{Z}}(\vec{\mathscr{X}}_{t_1+\ldots+t_s}[m],\mathbb{Q}).
\]
Hence, the columns of the lattice represent an infinite dimensional space of Kauffman finite-type invariants when $P$ is the Gaussian parity.
\end{theorem}
\begin{proof} Recall the definition of the $i$-th discrete derivative of a function $F:\mathbb{Z}^w \to G$, where $G$ is an abelian group.
\[
\partial^{0,\ldots,1,\ldots,0} F(x_1,\ldots,x_i,\ldots,x_w)=F(x_1,\ldots,x_i+1,\ldots,x_w)-F(x_1,\ldots,x_i,\ldots,x_w).
\]
Using ``equality of mixed partials'', we can define for any $w$-tuple of nonnegative integers $(a_1,\ldots,a_w)$ a derivative $\partial^{a_1,\ldots,a_w}$.  We may interpret $\theta_{t_1,\ldots, t_s}[m|k_1,\ldots,k_s]$ as a function $\mathbb{Z}^{(w,w')} \to \mathbb{Z}$, where $w=(a_1,\ldots,a_s)$ is the number of arrows of the form $\uparrow_{\oplus}^{k_1},\ldots,\uparrow_{\oplus}^{k_s}$ and $w'=(a_1',\ldots,a_s')$ is the number of arrows of the form $\uparrow_{\ominus}^{k_1},\ldots,\uparrow_{\ominus}^{k_s}$ (see \cite{CM} for more details).

Let $D(a_1,\ldots,a_s,a_1',\ldots,a_s')$ denote the subset of diagrams $D \in \vec{\mathscr{A}}[m]$ such that for each $i$, $1 \le i \le s$ , $D$ has exactly $a_i$ arrows of the form $\uparrow_{\oplus}^{k_i}$ and exactly $a_i'$ arrows of the form $\uparrow_{\ominus}^{k_i}$. The combinatorial formula is given by:
\begin{eqnarray*}
F_{t_1,t_2,\ldots,t_s}[m|k_1,k_2,\ldots k_s] &=& \sum_{\stackrel{(a_i,a_i')}{0 \le \sum a_i+a_i' \le \sum t_i}} c(a_i,a_i') \sum_{D \in D(a_i,a_i')} D ,\\
c(a_1,a_2,\ldots,a_s,a_1',a_2',\ldots,a_s') &=& \partial^{a_1,a_2,\ldots,a_s,a_1',a_2',\ldots,a_s'} \theta_{t_1,\ldots, t_s}[m|k_1,\ldots,k_s](\vec{0}).
\end{eqnarray*}
The proof that this really is a combinatorial formula in $\text{Hom}_{\mathbb{Z}}(\vec{\mathscr{X}}_{t_1+\ldots+t_s}[m],\mathbb{Q})$ follows exactly as in \cite{CM}.
\end{proof}
  
\section{Property \ref{prop5}: Bounds on Rank of $\vec{\mathscr{X}}_t[m]$} \label{prop5sec} In the present section, we compute bounds for the rank of each group in the lattice.  For simplicity, we consider only Gauss diagrams on $\mathbb{R}$. The upper bound is crude and is done using only combinatorial considerations.  A lower bound is found by considering the action on labelled diagrams by arrow ``flipping''.  

\subsection{$\Omega_t[m]$ and an upper bound on the rank} \label{omegatm} We compute a crude upper bound on the rank $\rho_t[m]$. Define $\Omega_t[m]$ by the following sum.

\begin{eqnarray*} 
\Omega_t[m] &=& 4m + (2t-1)!!2^t(m+1)^{t}-\sum_{j=1}^{\lfloor \frac{2t+1}{3} \rfloor} { 2 (t-j)+1 \choose j }(2(t-j)-1)!!2^{t} m^{t-j} \\
          &+& \sum_{k=2}^{t-1}(2k-1)!!2^{2k}(m+1)^k- \sum_{k=2}^{t-1}\sum_{j=1}^{\lfloor \frac{2k+1}{3} \rfloor} { 2 (k-j)+1 \choose j }(2(k-j)-1)!!2^{2k}m^{k-j}.
\end{eqnarray*}

\begin{lemma} The rank of $\vec{\mathscr{X}}_t[m]$ is not more than $\Omega_t[m]$.
\end{lemma}
\begin{proof} Note first that the chord diagrams on $\mathbb{R}$ having $k$ chords are in one-to-one correspondence with order 2 permutations on $2k$ letters which have no fixed points. As is well known, the number of such permutations is $(2k-1)!!$. The arrows may be signed and directed in $2^{2k}$ ways. If $k=t$, we may ignore the signs of the arrows.  

Note also that each chord diagram can be labelled in ${(m+1)}^k$ ways. When $k=1$, only $m$ of them do not vanish under a $\vec{Q1}[m]$ relation.

Finally suppose that you have chosen a chord diagram with $(k-j)$ chords where $j \le (2k+1)/3$. Also suppose that this diagram is signed, directed, and labelled by numbers between $1$ and $m$.  For this diagram, choose $j$ of the $2(k-j)+1$ intervals between the endpoints of the chords. In each of the $j$ intervals, we insert an arrow with adjacent endpoints and label it $m+1$ (see Figure \ref{irrelarrows}).  Each such diagram is trivial in $\vec{\mathscr{X}}_t[m]$, and hence, it does not contribute to the rank. Moreover, all of the diagrams constructed in this was are distinct. 

Accounting for all such unnecessary diagrams gives the formula for $\Omega_t[m]$ exactly as above. 
\end{proof}

\begin{figure}[h]
\scalebox{.8}{\psfig{figure=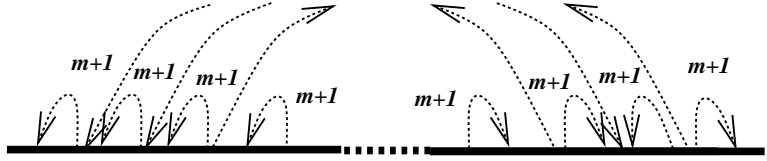}}
\caption{Different ways to add isolated arrows labelled $m+1$.} \label{irrelarrows}
\end{figure}

\subsection{Arrow Flipping and Chord Diagrams}\label{virtmodel} In the next four sections, we consider the action on $\vec{\mathscr{X}}_t[m]$ by arrow flipping and use it to find a lower bound on the rank of $\vec{\mathscr{X}}_t[m]$. Let $\overline{\mathscr{A}}[m]$ denote the set of signed dashed chord diagrams with labels from $1$ to $m+1$. We define the \emph{average map}, $\mu[m]:\mathbb{Z}[\overline{\mathscr{A}}[m]] \to \mathbb{Z}[\vec{\mathscr{A}}[m]]$ schematically as in \cite{Pol}:
\[
\mu[m] \left(\begin{array}{c} \scalebox{.4}{\psfig{figure=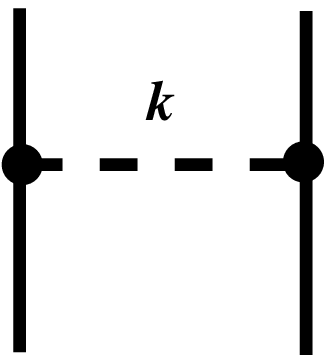}} \end{array} \right)= \begin{array}{c} \scalebox{.4}{\psfig{figure=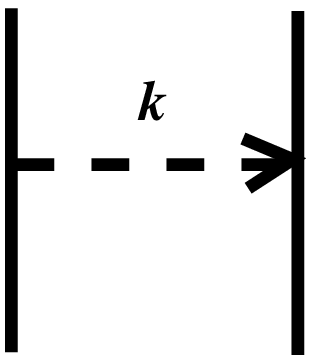}} \end{array}+\begin{array}{c} \scalebox{.4}{\psfig{figure=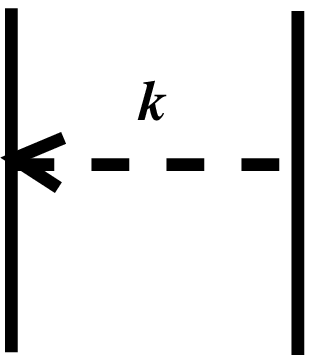}} \end{array}.
\]
The right hand side is a sum over all possible ways there are to direct the chords.  If $D$ has $t$ chords, then $\mu[m](D)$ is a sum of $2^t$ diagrams.  Also note that $\mu[m]$ preserves the sign and labels of each chord on the corresponding arrow. An important property of $\mu[m]$ is that it maps a diagram with $t$ chords to a  sum of diagrams having $t$ arrows each.

Given $\vec{D} \in \vec{\mathscr{A}}[m]$, denote by $\overline{D}$ the dashed chord diagram obtained from $\vec{D}$ by erasing all arrowheads.  This gives a map $\text{Bar}[m]:\vec{\mathscr{A}}[m] \to \overline{\mathscr{A}}[m]$. For a dashed signed chord diagram $D \in \overline{\mathscr{A}}[m]$ having $n$ chords, we have:
\[
\text{Bar}[m] \circ \mu[m](D)= 2^n \cdot D.
\]
It is important to note that both $\text{Bar}[m]$ and $\mu[m]$ preserve the number of arrows or chords of a diagram.

We define some relations on $\overline{\mathscr{A}}[m]$ using the map $\text{Bar}[m]:\mathbb{Z}[\vec{\mathscr{A}}[m]] \to \mathbb{Z}[\bar{\mathscr{A}}[m]]$ as follows:
\begin{eqnarray*}
\overline{\text{Q1}}[m] &=& \text{Bar}(\vec{\text{Q1}}[m]), \\
\overline{\text{Q2}}[m] &=& \text{Bar}(\vec{\text{Q2}}[m]), \\
\overline{\text{Q3}}[m] &=& \text{Bar}(\vec{\text{Q3}}[m]). \\
\end{eqnarray*}
Let $\overline{A}_t[m]=\text{Bar}[m](\vec{A}_t[m])$.  We define groups:
\[
\overline{\mathscr{X}}_t[m]=\frac{\mathbb{Z}[\overline{\mathscr{A}}[m]]}
{\left<\overline{\text{Q1}}[m],\overline{\text{Q2}}[m],\overline{\text{Q3}}[m],\overline{A}_t[m]\right> }.
\] 

The significance of this group can be described in terms of \emph{arrow flipping}. We will say that two diagrams $\vec{D}_1$ and $\vec{D}_2$ are equivalent by arrow flipping if $\vec{D}_2$ can be obtained from $\vec{D}_1$ by changing the direction of zero or more arrows of $\vec{D}_1$ (see Figure \ref{virtmv}). We will denote the resulting equivalence relation by $\sim$.
\begin{figure}[h]
\[
\begin{array}{cc}
\begin{array}{c} \scalebox{.25}{\psfig{figure=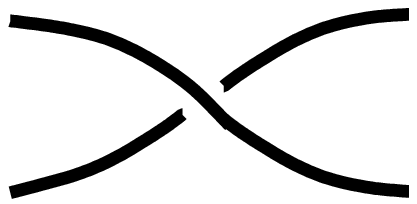}} \end{array} \leftrightarrow \begin{array}{c} \scalebox{.25}{\psfig{figure=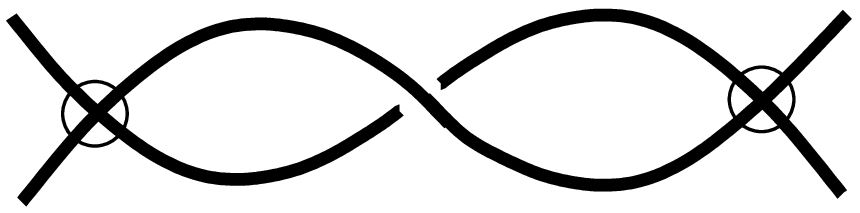}}
 \end{array},
&
\begin{array}{c} \scalebox{.25}{\psfig{figure=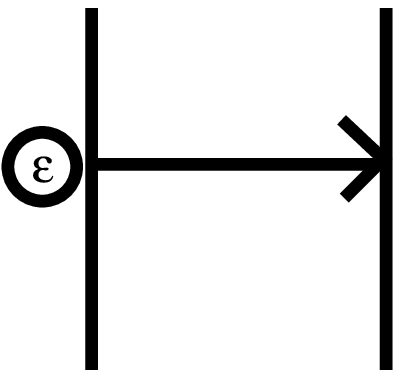}} \end{array} \leftrightarrow \begin{array}{c} \scalebox{.25}{\psfig{figure=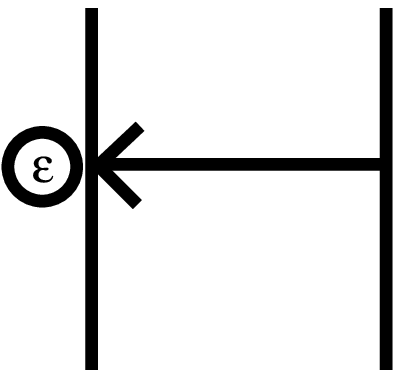}} \end{array}
\end{array}
\]
\caption{The virtualization move, arrow flipping.} \label{virtmv}
\end{figure}
\begin{lemma}[Arrow Flipping] The quotient of $\vec{\mathscr{X}}_t[m]$ by the action of arrow flipping (diagrams in $\mathscr{A}[m]$) is isomorphic to $\overline{\mathscr{X}}_t[m]$.
\end{lemma}
\begin{proof} The fibers of the surjection $\text{Bar}[m]:\vec{\mathscr{X}}_t[m] \to \overline{\mathscr{X}}_t[m]$ are the equivalence classes of $\sim$.
\end{proof}

When $P$ is a parity of flat virtual knots, the dual spaces of the groups $\overline{\mathscr{X}}_t[m]$ yield virtual knot invariants which are invariant under the virtualization move (see Figure \ref{virtmv}). This is the content of the next theorem. 

\begin{theorem} Let $P$ be a parity of flat virtual knots.  If $v \in \text{Hom}_{\mathbb{Z}}(\overline{\mathscr{X}}_t[m],\mathbb{Q})$, then $v \circ \text{Bar}[m] \circ I[m] \circ \Lambda[m]:\mathbb{Z}[\mathscr{D}] \to \mathbb{Q}$ is a Kauffman finite-type invariant of order $\le t$ which is invariant under the virtualization move.
\end{theorem}
\begin{proof} Note that $\text{Bar}:\vec{\mathscr{X}}_t[m] \to \overline{\mathscr{X}}_t[m]$ is a surjection. Hence, $v$ can be identified as an element $\vec{v}[m] \in \text{Hom}_{\mathbb{Z}}(\vec{\mathscr{X}}_t[m],\mathbb{Q})$.  In particular, we have $v\circ \text{Bar}[m]=\vec{v}[m]$.  For any $D,D' \in \vec{\mathscr{A}}[m]$ which differ by the direction of some arrows, we have that $\text{Bar}[m](D)=\text{Bar}[m](D')$.  Hence,
\[
\vec{v}[m](D)=v \circ \text{Bar}[m](D)=v \circ \text{Bar}[m](D')=\vec{v}[m](D').
\]
Let $E$ be a Gauss diagram of a virtual knot. Since $P$ is a parity of flat virtual knots, $f$ assigns the same label to every diagram equivalent to $E$ by changing the direction of an arrow. This proves the theorem by definition of the virtualization move.
\end{proof}

It follows from Theorem \ref{thetacomb} that the invariants $\theta_{t_1,\ldots,t_s}[m|k_1,\ldots,k_s]$ are in the image of $\text{Bar}^*[m]$, the dual of the surjection $\vec{\mathscr{X}}_t[m]\to\overline{\mathscr{X}}_t[m]\to 0$. Indeed, the combinatorial formula is unchanged by changing the direction of any arrow.

\subsection{Algebraic Structure of $\vec{\mathscr{X}}_t[m]$ and $\overline{\mathscr{X}}_t[m]$} In this section, we investigate the relations between the sequence of groups $\vec{\mathscr{X}}_t[m]$ and $\overline{\mathscr{X}}_t[m]$ in $m$ and $t$. We establish Lemma \ref{sixtermmap} which allows the rank of $\vec{\mathscr{X}}_t[m]$ to be underestimated. We write $\mathscr{X}_t[m]$ for $\vec{\mathscr{X}}_t[m]$ or $\overline{\mathscr{X}}_t[m]$. We determine the structure of these groups simultaneously. First, we have the following sequence of surjections:
\[
\xymatrix{
\ldots \ar[r] & \mathscr{X}_t[m] \ar[r] & \mathscr{X}_{t-1}[m] \ar[r] & \ldots \ar[r] & \mathscr{X}_2[m] \ar[r] & \mathscr{X}_1[m]
}.
\] 
Denote the kernel of the surjection by $\text{Ker}_t[m]$. This gives a short exact sequence:
\[
\xymatrix{
0 \ar[r] & \text{Ker}_t[m] \ar[r] & \mathscr{X}_t[m] \ar[r] & \mathscr{X}_{t-1}[m] \ar[r] & 0 
}.
\]
To understand this quotient, we introduce the $f^m$-labelled versions of the six term, one term, and sign relations for signed arrow diagrams.
\[
\underline{\vec{\text{1T}}_{\pm t}[m]:} \,\, \begin{array}{c} \scalebox{.33}{\psfig{figure=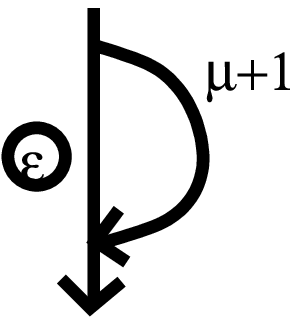}} \end{array}=0, \,\,\,\, \underline{\vec{\text{NS}}_{\pm t}[m]:} \,\, \begin{array}{c} \scalebox{.33}{\psfig{figure=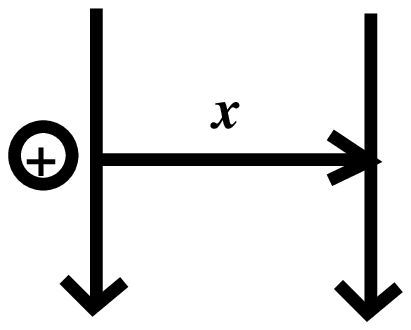}} \end{array}+\begin{array}{c} \scalebox{.33}{\psfig{figure=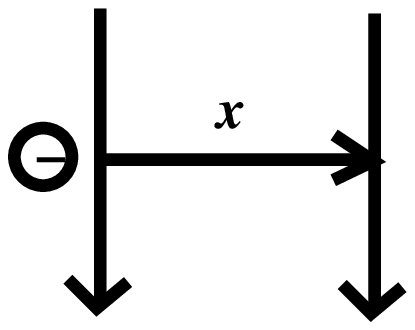}} \end{array}=0,
\]
\[
\underline{\vec{\text{6T}}_{\pm t}[m]:} \,\, \begin{array}{c} \scalebox{.33}{\psfig{figure=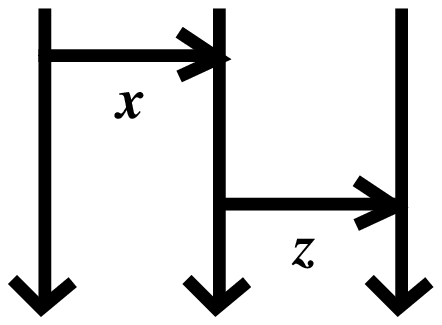}} \end{array}-\begin{array}{c} \scalebox{.33}{\psfig{figure=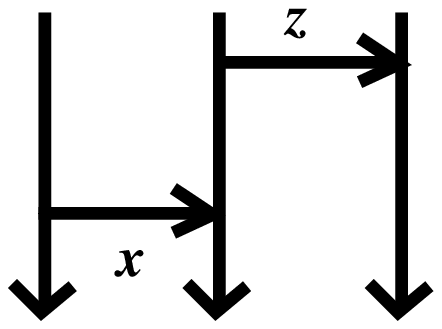}} \end{array}+ \begin{array}{c} \scalebox{.33}{\psfig{figure=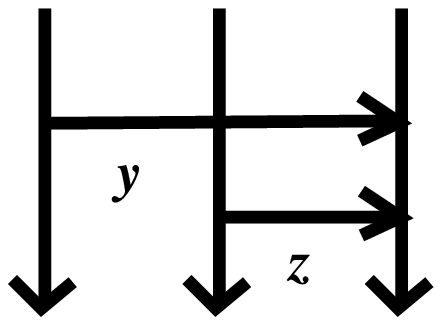}} \end{array}-\begin{array}{c} \scalebox{.33}{\psfig{figure=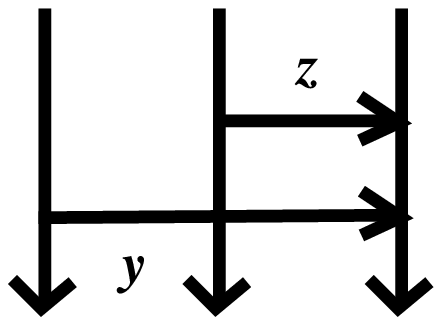}} \end{array} 
+\begin{array}{c} \scalebox{.33}{\psfig{figure=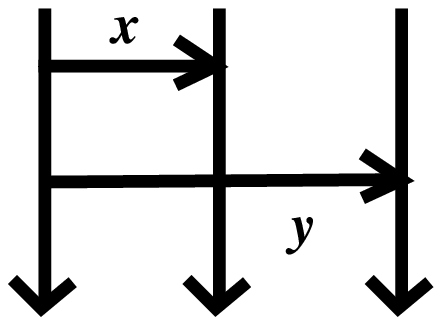}} \end{array}-\begin{array}{c} \scalebox{.33}{\psfig{figure=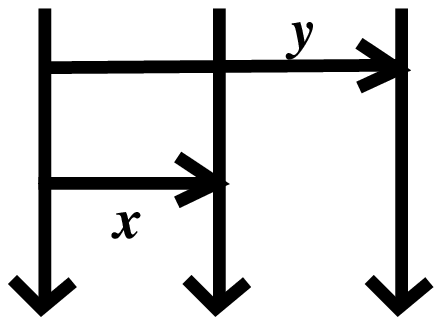}} \end{array}=0.
\]

In the $\vec{\text{6T}}_{\pm t}[m]$ relation, all drawn arrows have the same sign.  Also, it is required that either $x>y=z$, $y>x=z$, $z>x=y$ or $x=y=z=m+1$.

For each of these relations $\vec{R}$ there are also images of these relations under the map $\text{Bar}$.  These images are denoted by $\overline{R}$.  For example, $\overline{\text{6T}}_{\pm t}[m]=\text{Bar}(\vec{\text{6T}}_{\pm t}[m])$. If $R$ is written instead of $\vec{R}$ or $\overline{R}$, the given statement holds true in either case. For example, we write $\text{6T}_{\pm t}[m]$ for either $\vec{\text{6T}}_{\pm t}[m]$ or $\overline{\text{6T}}_{\pm t}[m]$ when the statement is true for either. Also, if instead of $\pm t$ we write $|t|$ in a relation, we mean that the signs of the arrows in the relation are to be erased.

\begin{lemma}\label{sixtermmap} Let $\mathscr{A}_{|t|}[m]$ ($=\vec{\mathscr{A}}_{|t|}[m]$ or $\overline{\mathscr{A}}_{|t|}[m]$) denote the free abelian group generated by those Gauss diagrams with labels up to $m+1$ having exactly $t$ unsigned arrows.  Then the following sequence is exact:
\[
\xymatrix{0 \ar[r] & \text{Hom}_{\mathbb{Z}}\left(\mathscr{X}_{t-1}[m] ,\mathbb{Q}\right) \ar[r] & \text{Hom}_{\mathbb{Z}}\left(\mathscr{X}_t[m] ,\mathbb{Q}\right) \ar[r] & \text{Hom}_{\mathbb{Z}}\left(\frac{\mathscr{A}_{|t|}[m]}{\left<\text{6T}_{|t|}[m],\text{1T}_{|t|}[m]\right>} ,\mathbb{Q}\right)}.
\]
\end{lemma}
\begin{proof} The idea is similar to that in \cite{Pol}.  Full details of an analogous argument are in \cite{C2}. 
\end{proof}

\subsection{$f^m$-Labelled Four Term and Six Term Relations}\label{foursix} \label{foursix} To estimate the rank, we require an additional relation known as the four-term relation (compare \cite{C2}).  In this section, we define the $f^m$-labelled version of the \text{4T} relation and investigate its properties.

In our case, $\overline{\text{4T}}_{|t|}[m] \in \overline{\mathscr{A}}_{|t|}[m]$ for a given $t$ and $m$. Pictorially, it is given as below.
\begin{eqnarray*}
\underline{\overline{\text{4T}}_{|t|}[m]:}\begin{array}{c} \scalebox{.33}{\psfig{figure=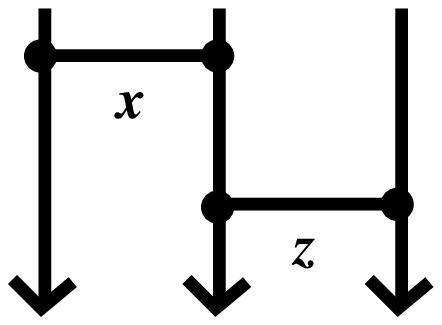}} \end{array}-\begin{array}{c} \scalebox{.33}{\psfig{figure=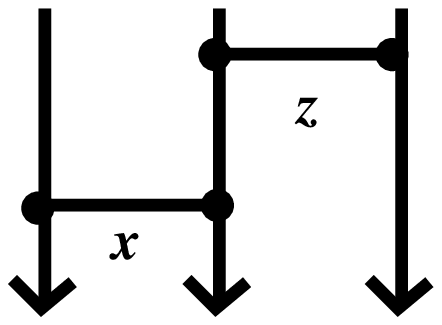}} \end{array} &=& \begin{array}{c} \scalebox{.33}{\psfig{figure=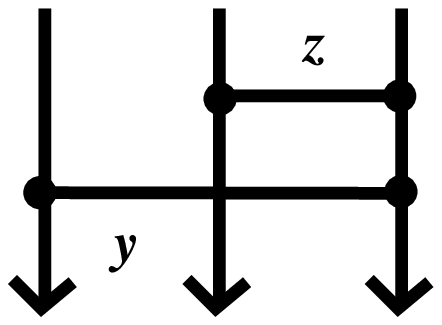}} \end{array}-\begin{array}{c} \scalebox{.33}{\psfig{figure=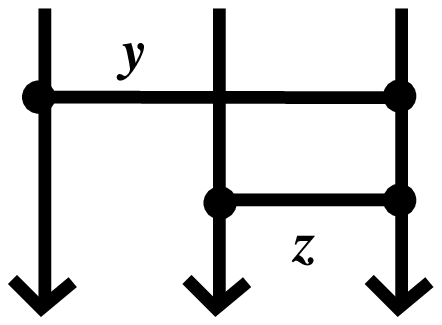}} \end{array}  \\
&=& \begin{array}{c} \scalebox{.33}{\psfig{figure=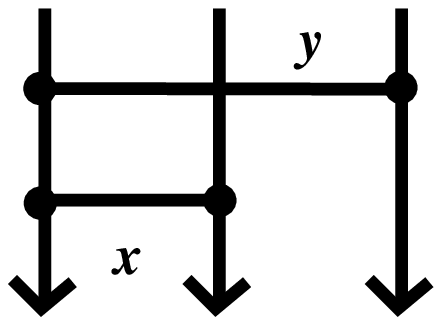}} \end{array}-\begin{array}{c} \scalebox{.33}{\psfig{figure=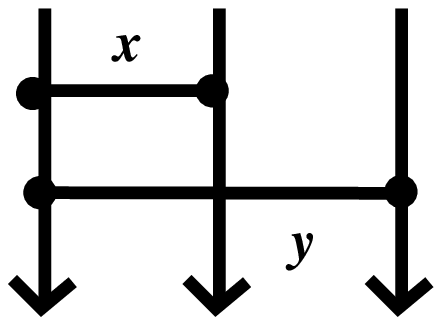}} \end{array}.
\end{eqnarray*} 
We require as usual that either $x>y=z$, $y>x=z$, $z>x=y$, or $x=y=z=m+1$.  For a fixed embedding of the three vertical strands into a chord diagram, we describe an $f^m$-labelled version of a notation originally due to Polyak \cite{Pol}. We denote by $\alpha_{ij}^x$ the undirected chord between strands $i$ and $j$ with label $x$.  Specific groupings of chords are denoted as follows:
\begin{eqnarray*}
\begin{array}{c} \scalebox{.33}{\psfig{figure=fourtern1.eps}} \end{array}-\begin{array}{c} \scalebox{.33}{\psfig{figure=fourtern2.eps}} \end{array} &=& [\alpha_{12}^x,\alpha_{23}^z], \\
\begin{array}{c} \scalebox{.33}{\psfig{figure=fourtern3.eps}} \end{array}-\begin{array}{c} \scalebox{.33}{\psfig{figure=fourtern4.eps}} \end{array} &=& [\alpha_{13}^y,\alpha_{23}^z], \\
\begin{array}{c} \scalebox{.33}{\psfig{figure=fourtern5.eps}} \end{array}-\begin{array}{c} \scalebox{.33}{\psfig{figure=fourtern6.eps}} \end{array} &=& [\alpha_{12}^x,\alpha_{13}^y]. \\
\end{eqnarray*}
If $[a,b]$ is one of the terms given immediately above, define $[b,a]=-[a,b]$.  We denote by $a_{ij}^x$ the \emph{arrow} directed from strand $i$ to strand $j$ having label $x$. Using this notation, we may write the $\vec{\text{6T}}_{|t|}[m]$ relation:
\[
\left[a_{12}^x, a_{23}^z \right]+\left[ a_{12}^x, a_{13}^y\right]+\left[a_{13}^y, a_{23}^z \right]=0.
\]
\begin{lemma} \label{brackrel}Let $\sigma \in \mathbb{S}_3$ denote a permutation of the three vertical intervals of a $\text{6T}_{|t|}[m]$ relation.  Then the following relation holds:
\[
\left[a_{\sigma(1) \sigma(2)}^x, a_{\sigma(2) \sigma(3)}^z \right]+\left[ a_{\sigma(1) \sigma(2)}^x, a_{\sigma(1) \sigma(3)}^y\right]+\left[a_{\sigma(1) \sigma(3)}^y, a_{\sigma(2) \sigma(3)}^z \right]=0.
\]
\end{lemma}
\begin{proof} The relation is true when $\sigma=1$.  Consider the permutation $\sigma=( 2 \, 3)$ which is written in cycle notation. This corresponds to fixing the first string and interchanging the second and third strings. In terms of a diagram, this can be written as follows:
\[
\begin{array}{c} \scalebox{.33}{\psfig{figure=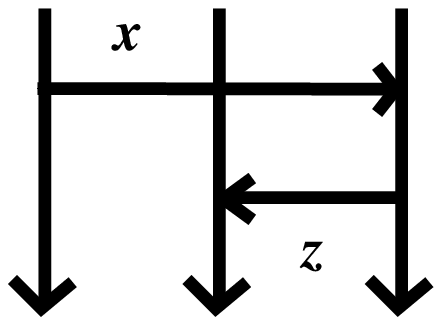}} \end{array}-\begin{array}{c} \scalebox{.33}{\psfig{figure=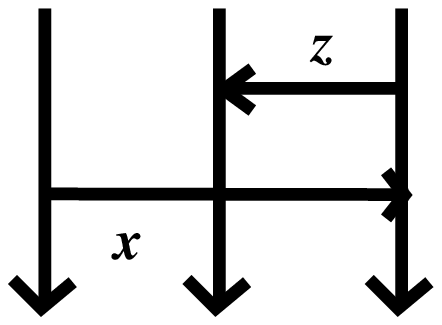}} \end{array}+ \begin{array}{c} \scalebox{.33}{\psfig{figure=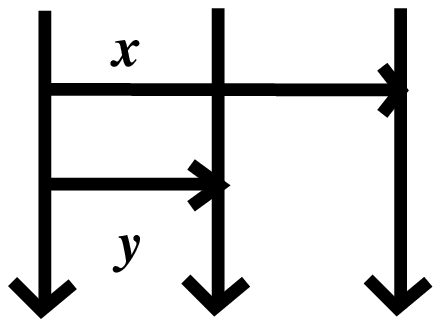}} \end{array}-\begin{array}{c} \scalebox{.33}{\psfig{figure=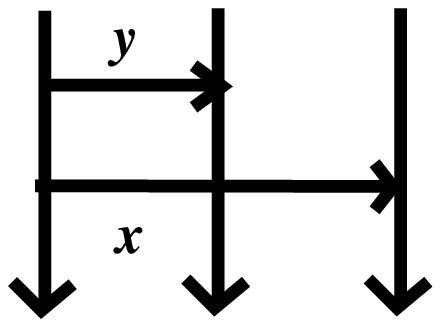}} \end{array} 
+\begin{array}{c} \scalebox{.33}{\psfig{figure=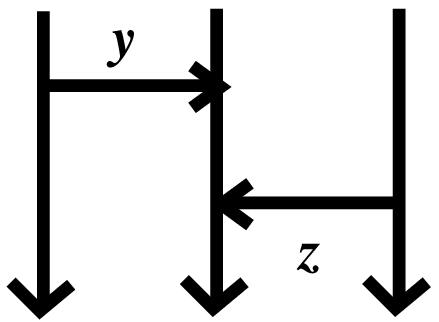}} \end{array}-\begin{array}{c} \scalebox{.33}{\psfig{figure=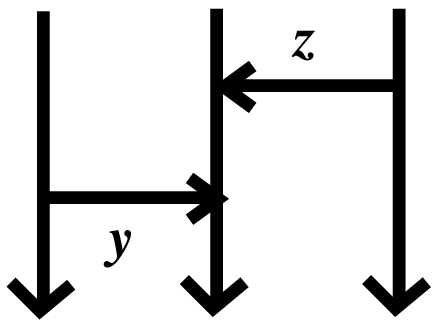}} \end{array}=0.
\]
Each of the three pairs $D-E$ in the above diagram matches one of the terms in the directed bracket notation. It is then simply a matter of writing out the relation to check that it works:
\begin{eqnarray*}
\left[a_{13}^x, a_{32}^z \right]-\left[ a_{12}^y, a_{13}^x\right]+\left[a_{12}^y, a_{32}^z \right] &=& \left[a_{13}^x, a_{32}^z \right]+\left[ a_{13}^x, a_{12}^y\right]+\left[a_{12}^y, a_{32}^z \right] \\ 
&=& \left[a_{\sigma(1) \sigma(2)}^x, a_{\sigma(2) \sigma(3)}^z \right]+\left[ a_{\sigma(1) \sigma(2)}^x, a_{\sigma(1) \sigma(3)}^y\right]+\left[a_{\sigma(1) \sigma(3)}^y, a_{\sigma(2) \sigma(3)}^z \right].\\ 
\end{eqnarray*}
The other four cases follow similarly.  For the reader's convenience, the bracket notation for all six relations are given together below:
\begin{eqnarray} 
\label{brackrel1}
\sigma=(1)(2)(3), \,\, \left[a_{12}^x, a_{23}^z \right]+\left[ a_{12}^x, a_{13}^y\right]+\left[a_{13}^y, a_{23}^z \right] &=& 0, \\\label{brackrel2}
\sigma=(23), \,\, \left[a_{13}^x, a_{32}^z \right]-\left[ a_{12}^y, a_{13}^x\right]+\left[a_{12}^y, a_{32}^z \right] &=& 0, \\ \label{brackrel3}
\sigma=(123), \,\,-\left[a_{31}^z, a_{23}^x \right]-\left[ a_{21}^y, a_{23}^x\right]+\left[a_{21}^y, a_{31}^z \right] &=& 0, \\ \label{brackrel4}
\sigma=(13), \,\,-\left[a_{21}^z, a_{32}^x \right]-\left[ a_{31}^y, a_{23}^x\right]-\left[a_{21}^z, a_{31}^y \right] &=& 0, \\ \label{brackrel5}
\sigma=(12), \,\,\left[a_{21}^x, a_{13}^z \right]+\left[ a_{21}^x, a_{23}^y\right]-\left[a_{13}^z, a_{23}^y \right] &=& 0, \\ \label{brackrel6}
\sigma=(321), \,\,-\left[a_{12}^z, a_{31}^x \right]+\left[ a_{31}^x, a_{32}^y\right]-\left[a_{12}^z, a_{32}^y \right] &=& 0.
\end{eqnarray}
This completes the proof of the lemma. 
\end{proof}
\begin{lemma} \label{switchrel} The following relations hold in $\vec{\mathscr{A}}_{|t|}[m]/\left< \vec{\text{6T}}_{|t|}[m],\vec{\text{1T}}_{|t|}[m] \right>$, where either $x>y=z$, $y>x=z$, $z>x=y$, or $x=y=z=m+1$.
\begin{enumerate}
\item $\left[a_{12}^x, a_{23}^z\right]+\left[a_{12}^x, a_{32}^z\right]=\left[a_{32}^z,  a_{13}^y\right]+\left[a_{23}^z,a_{13}^y\right]$
\item $\left[a_{12}^x, a_{23}^z\right]+\left[a_{21}^x, a_{23}^z\right]=\left[a_{13}^y,  a_{12}^x\right]+\left[a_{13}^y, a_{21}^x\right]$
\item $\left[a_{21}^y, a_{32}^x\right]+\left[a_{21}^y , a_{23}^x\right]=\left[a_{32}^x,  a_{31}^z\right]+\left[a_{23}^x,a_{31}^z\right]$
\item $\left[a_{21}^z, a_{32}^x\right]+\left[a_{12}^z, a_{32}^x\right]=\left[a_{31}^y,  a_{21}^z\right]+\left[a_{31}^y,a_{12}^z\right]$
\end{enumerate}
\end{lemma}
\begin{proof} We prove the first relation only.  We write out relations (2) and (3) from the proof of Lemma \ref{brackrel} using distinct indices.
\begin{eqnarray*}
\left[a_{12}^x, a_{23}^z \right]+\left[ a_{12}^x, a_{13}^y\right]+\left[a_{13}^y, a_{23}^z \right] &=& 0, \\
\left[a_{13}^i, a_{32}^k \right]-\left[ a_{12}^j, a_{13}^i\right]+\left[a_{12}^j, a_{32}^k \right] &=& 0.
\end{eqnarray*}
Set $j=x$, $i=y$, $k=z$.  The middle terms cancel out when they are added together.  The resulting expression can be rearranged using the identity $[b,a]=-[a,b]$ to obtain the first relation above.
\end{proof}
The following lemma is the $f^m$-labelled version of a theorem of Polyak (see \cite{Pol}).
\begin{lemma} The average map $\mu[m]$ satisfies the following properties.
\begin{enumerate}
\item $\mu[m]\left(\left< \overline{\text{1T}}_{|t|}[m] \right> \right) \subset \left< \vec{\text{1T}}_{|t|}[m] \right>$
\item $\mu[m]\left(\left< \overline{\text{4T}}_{|t|}[m] \right> \right) \subset \left< \vec{\text{6T}}_{|t|}[m] \right>$
\item The average map descends to the quotient:
\[
\mu[m]:\frac{\overline{\mathscr{A}}_{|t|}[m]}{\left<\overline{\text{4T}}_{|t|}[m], \overline{\text{1T}}_{|t|}[m] \right>} \to \frac{\vec{\mathscr{A}}_{|t|}[m]}{\left<\vec{\text{6T}}_{|t|}[m], \vec{\text{1T}}_{|t|}[m] \right>}.
\]
\end{enumerate}
\end{lemma}
\begin{proof} The first claim is clear from the definitions.  For the second claim, we compute $\mu[m]$ and apply relations from Lemma \ref{switchrel}. If $x \ne z$, set $y=\min\{x,z\}$.  If $x=z<m+1$, any $y>x$ will do.  If $x=z=m+1$, set $y=m+1$. The following computations are all performed in the quotient group $\vec{\mathscr{A}}_{|t|}[m]/\left<\vec{\text{6T}}_{|t|}[m]\right>$.
\begin{eqnarray*}
\mu[m]\left( \left[\alpha_{12}^x,\alpha_{23}^z \right] \right) &=& \sum \left[a_{12}^x,a_{23}^z \right]+\left[a_{21}^x,a_{23}^z \right]+\left[a_{12}^x,a_{32}^z \right]+\left[a_{21}^x,a_{32}^z \right] \\
&=& \sum \left[a_{13}^y,a_{12}^x \right]+\left[a_{13}^y,a_{21}^x \right]+\left[a_{31}^y,a_{21}^x \right]+\left[a_{31}^y,a_{12}^x \right] \\
&=& \mu[m]\left( \left[\alpha_{13}^y,\alpha_{12}^x \right] \right).
\end{eqnarray*}
Here the sum is taken over all fixed choices of the directions of the arrows outside three drawn vertical intervals. It follows that in the quotient group,
\[
\mu[m]\left( \left[\alpha_{12}^x,\alpha_{23}^z \right]-\left[\alpha_{13}^y,\alpha_{12}^x \right] \right)=0.
\]
This proves a case of the result.  The other cases follow by applying different identities from Lemma \ref{switchrel}.
\end{proof}

\subsection{Proof of Property \ref{prop5}:} We establish a lower bound on the rank of $\vec{\mathscr{X}}_t[m]$ by computing the rank of the free group $\text{Hom}_{\mathbb{Z}}(\overline{\mathscr{X}}_t[m],\mathbb{Q})$. In the sequel, we will make frequent use of the following isomorphism (see \cite{S}). Let $M$ be any $\mathbb{Z}$-module.
\[
\text{Hom}_{\mathbb{Z}} \left(M,\mathbb{Q}\right) \cong \text{Hom}_{\mathbb{Z}} \left(M,\text{Hom}_{\mathbb{Z}}(\mathbb{Q},\mathbb{Q})\right) \cong \text{Hom}_{\mathbb{Z}} \left(\mathbb{Q} \otimes M,\mathbb{Q}\right) 
\]
The crux of the proof revolves around the following commutative diagram:
\[
\xymatrix{
\mathbb{Q} \otimes \frac{\overline{\mathscr{A}}_{|t|}[m]}{\left<\overline{\text{4T}}_{|t|}[m], \overline{\text{1T}}_{|t|}[m] \right>}\ar[rr]^{2^t \mu'[m]} \ar[dr]_{\mu[m]} & & \mathbb{Q} \otimes \frac{\overline{\mathscr{A}}_{|t|}[m]}{\left<\overline{\text{6T}}_{|t|}[m], \overline{\text{1T}}_{|t|}[m] \right>} \\
& \mathbb{Q} \otimes \frac{\vec{\mathscr{A}}_{|t|}[m]}{\left<\vec{\text{6T}}_{|t|}[m], \vec{\text{1T}}_{|t|}[m] \right>} \ar[ur]_{\text{Bar}[m]} & \\
}.
\]
where $\mu'[m]=\frac{1}{2^t} \text{Bar}[m]\circ \mu[m]$. We note that for any $D \in \mathbb{Q} \otimes  \frac{\overline{\mathscr{A}}_{|t|}[m]}{\left<\overline{\text{4T}}_{|t|}[m], \overline{\text{1T}}_{|t|}[m] \right>}$, $\mu'[m](D)=D$ and hence $\mu'[m]$ is a surjection.

In addition, we have the following \emph{two-term} or \emph{commutativity relations}:
\[
\underline{\overline{\text{2T}}_{|t|}[m]:}\begin{array}{c} \scalebox{.33}{\psfig{figure=fourtern3.eps}} \end{array}=\begin{array}{c} \scalebox{.33}{\psfig{figure=fourtern4.eps}} \end{array}.
\]
Here, $y$ and $z$ are any two labels $\le m+1$. 

\begin{lemma} \label{sixtermtwo} There is an isomorphism of groups:
\[
\mathbb{Q} \otimes \frac{\overline{\mathscr{A}}_{|t|}[m]}{\left<\overline{\text{6T}}_{|t|}[m], \overline{\text{1T}}_{|t|}[m] \right>} \cong \mathbb{Q} \otimes \frac{\overline{\mathscr{A}}_{|t|}[m]}{\left<\overline{\text{2T}}_{|t|}[m], \overline{\text{1T}}_{|t|}[m]\right>}.
\]
\end{lemma}
\begin{proof} Since $\mu'[m]$ is a surjection, it follows that $\mathbb{Q} \otimes  \frac{\overline{\mathscr{A}}_{|t|}[m]}{\left<\overline{\text{6T}}_{|t|}[m], \overline{\text{1T}}_{|t|}[m] \right>}$ is a homomorphic image of $ \mathbb{Q} \otimes \frac{\overline{\mathscr{A}}_{|t|}[m]}{\left<\overline{\text{4T}}_{|t|}[m], \overline{\text{1T}}_{|t|}[m] \right>}$.  Since the four term relations are satisfied in the second group, they must also be satisfied in the first. Therefore, both the four term and six term relations are satisfied in $\mathbb{Q} \otimes \frac{\overline{\mathscr{A}}_{|t|}[m]}{\left<\overline{\text{6T}}_{|t|}[m], \overline{\text{1T}}_{|t|}[m]\right>}$.

Suppose that $y,z$ are given, $1 \le y,z \le m+1$.  If $y=z=m+1$, set $x=m+1$.  If $y \ne z$, set $x= \text{min}\{y,z \}$.  If $y=z<m+1$, choose any $x$ satisfying $z <x \le m+1$.  Then we write out the corresponding undirected six-term relation:  
\[
\underbrace{\begin{array}{c} \scalebox{.33}{\psfig{figure=fourtern1.eps}} \end{array}-\begin{array}{c} \scalebox{.33}{\psfig{figure=fourtern2.eps}} \end{array}+
\begin{array}{c} \scalebox{.33}{\psfig{figure=fourtern5.eps}} \end{array}-\begin{array}{c} \scalebox{.33}{\psfig{figure=fourtern6.eps}} \end{array}}_{\overline{\text{4T}}_{|t|}[m]} +
\underbrace{\begin{array}{c} \scalebox{.33}{\psfig{figure=fourtern3.eps}} \end{array}-\begin{array}{c} \scalebox{.33}{\psfig{figure=fourtern4.eps}} \end{array}}_{\overline{\text{2T}}_{|t|}[m]} =0.
\]
Since the left bracketed expression vanishes, it follows that the right bracket expression is also zero.  Hence, all commutativity relations are satisfied.

Now, the six term relation $\overline{\text{6T}}_{|t|}[m]$ can be written as a sum of three two term relations $\overline{\text{2T}}_{|t|}[m]$.  This proves the lemma.
\end{proof}

It follows that the equivalence classes of $\mathbb{Q} \otimes \frac{\overline{\mathscr{A}}_{|t|}[m]}{\left<\overline{\text{2T}}_{|t|}[m], \overline{\text{1T}}_{|t|}[m]\right>}$ may be identified with monomials in the variables $x_1,\ldots,x_m$ with total degree $t$:
\[
x_1^{t_1} \cdot x_2^{t_2} \cdot \ldots \cdot x_m^{t_m}.
\]
Indeed, since the commutativity relations are satisfied, diagrams with the same number of arrows with the same labellings are necessarily equivalent. Likewise, diagrams with an arrow labelled $m+1$ are equivalent to a one term relation.  Note that the number of such monomials is given by the so-called multinomial coefficient \cite{B}.  

Under this correspondence, the map in Lemma \ref{sixtermmap} sends $F_{t_1,\ldots,t_s}[m|k_1,\ldots k_s]$ to the monomial $t_1! t_2 ! \cdot \ldots \cdot t_s! x_{k_1}^{t_1} \cdot\ldots \cdot x_{k_s}^{t_s}$. It follows that the rightmost map in Lemma \ref{sixtermmap} is a surjection and that the sequence extends to a short exact sequence. Also, our argument has shown that the rightmost group in Lemma \ref{sixtermmap} is free. Thus, the short exact sequence splits.  It follows by induction that the rank of $\text{Hom}_{\mathbb{Z}}(\overline{\mathscr{X}}_t[m],\mathbb{Q})$ is the sum of multinomial coefficients and the rank of $\text{Hom}_{\mathbb{Z}}(\left<\underline{\hspace{1cm}}\right>,\mathbb{Q})$ (i.e. the dual space of the diagram containing no arrows):
\[
1+\sum_{k=1}^t \left( \hspace{-.2cm} \left(\begin{array}{c} m \\ k \end{array} \right) \hspace{-.2cm} \right)= 1+\sum_{k=1}^t \left(\begin{array}{c} m+k-1\\ k \end{array}\right)=1+\frac{t+1}{m}\left(\begin{array}{c}m+t \\ t+1 \end{array}\right)-1.
\]	
The last equality follows from a computation in \emph{Mathematica}. This establishes the lower bound on the rank.
\bibliographystyle{plain}
\bibliography{bib_comm}

\end{document}